\documentclass[10pt,reqno,oneside]{amsproc}
\title[Boundary controllability of 3D ideal MHD]{Exact boundary controllability of the 3D incompressible ideal MHD system}
\author[I.~Kukavica]{Igor Kukavica}
\address{Department of Mathematics, University of Southern California, Los Angeles, CA 90089}
\email{kukavica@usc.edu}
\author[W.~O\.za\'nski]{Wojciech O\.za\'nski}
\address{Department of Mathematics, Florida State University, Tallahassee, FL 32306}
\email{wozanski@fsu.edu}
  \chardef\forshowkeys=0
  \chardef\refcheck=0
  \chardef\showllabel=0
  \chardef\sketches=0
\usepackage{enumitem}
\usepackage{datetime}
\usepackage{fancyhdr}
\usepackage{comment}
\allowdisplaybreaks
\ifnum\forshowkeys=1

  \usepackage[notref,notcite,color]{showkeys}
\fi
\usepackage[margin=1in]{geometry}
\usepackage{amsmath, amsthm, amssymb, mathtools, caption}
\usepackage{times}
\usepackage{graphicx}
\usepackage[usenames,dvipsnames,svgnames,table]{xcolor}
\usepackage{marginnote}
\usepackage[unicode,breaklinks=true,colorlinks=true,linkcolor=blue,urlcolor=blue,citecolor=blue]{hyperref}
\usepackage[most]{tcolorbox}
\usepackage{tikz}

\ifnum\refcheck=1
  \usepackage{refcheck}
\fi
\begin{document}
\def\YY{X}
\def\na{\nabla }
\def\oo{{\Omega'}}
\def\ooo{\overline{\Omega'}}
\def\wo{{\Omega}}
\def\owo{\overline{\Omega}}
\def\OO{\mathcal O}
\def\SS{\mathbb S}
\def\CC{\mathbb C}
\def\RR{\mathbb R}
\def\TT{\mathbb T}
\def\ZZ{\mathbb Z}
\def\HH{\mathbb H}
\def\RSZ{\mathcal R}
\def\LL{\mathcal L}
\def\SL{\LL^1}
\def\ZL{\LL^\infty}
\def\GG{\mathcal G}
\def\tt{\langle t\rangle}
\def\erf{\mathrm{Erf}}
\def\mgt#1{\textcolor{magenta}{#1}}
\def\ff{\rho}
\def\gg{G}
\def\sqrtnu{\sqrt{\nu}}
\def\ww{w}
\def\ft#1{#1_\xi}
\def\lec{\lesssim}
\def\gec{\gtrsim}
\renewcommand*{\Re}{\ensuremath{\mathrm{{\mathbb R}e\,}}}
\renewcommand*{\Im}{\ensuremath{\mathrm{{\mathbb I}m\,}}}
\ifnum\showllabel=1
 \def\llabel#1{\marginnote{\color{lightgray}\rm\small(#1)}[-0.0cm]\notag}
\else
 \def\llabel#1{\notag}
\fi

\newcommand{\eqnb}{\begin{equation}}
\newcommand{\eqne}{\end{equation}}

\newcommand{\nn}{\mathsf{n}}
\newcommand{\norm}[1]{\left\|#1\right\|}
\newcommand{\nnorm}[1]{\lVert #1\rVert}
\newcommand{\abs}[1]{\left|#1\right|}
\newcommand{\NORM}[1]{|\!|\!| #1|\!|\!|}
\newtheorem{theorem}{Theorem}[section]
\newtheorem{Theorem}{Theorem}[section]
\newtheorem{corollary}[theorem]{Corollary}
\newtheorem{Corollary}[theorem]{Corollary}
\newtheorem{proposition}[theorem]{Proposition}
\newtheorem{Proposition}[theorem]{Proposition}
\newtheorem{Lemma}[theorem]{Lemma}
\newtheorem{lemma}[theorem]{Lemma}
\theoremstyle{definition}
\newtheorem{definition}{Definition}[section]
\newtheorem{Remark}[Theorem]{Remark}
\def\theequation{\thesection.\arabic{equation}}
\numberwithin{equation}{section}
\definecolor{mygray}{rgb}{.6,.6,.6}
\definecolor{myblue}{rgb}{9, 0, 1}
\definecolor{colorforkeys}{rgb}{1.0,0.0,0.0}
\newlength\mytemplen
\newsavebox\mytempbox
\makeatletter
\newcommand\mybluebox{%
    \@ifnextchar[
       {\@mybluebox}%
       {\@mybluebox[0pt]}}
\def\@mybluebox[#1]{%
    \@ifnextchar[
       {\@@mybluebox[#1]}%
       {\@@mybluebox[#1][0pt]}}
\def\@@mybluebox[#1][#2]#3{
    \sbox\mytempbox{#3}%
    \mytemplen\ht\mytempbox
    \advance\mytemplen #1\relax
    \ht\mytempbox\mytemplen
    \mytemplen\dp\mytempbox
    \advance\mytemplen #2\relax
    \dp\mytempbox\mytemplen
    \colorbox{myblue}{\hspace{1em}\usebox{\mytempbox}\hspace{1em}}}
\makeatother
\def\bnew{\colr }
\def\enew{\colb {}}
\def\bold{\colu }
\def\eold{\colb {}}
\def\phiij{\phi_{ij}}
\def\un{u^{(n)}}
\def\Bn{B^{(n)}}
\def\unp{u^{(n+1)}}
\def\unm{u^{(n-1)}}
\def\Bnp{B^{(n+1)}}
\def\Bnm{B^{(n-1)}}
\def\pn{p^{(n)}}
\def\pnm{p^{(n-1)}}

\def\ee{\epsilon}
\def\eeo{\overline{\epsilon}}
\def\eet{\bar\epsilon}
\def\us{U}
\def\rr{r}
\def\weaks{\text{\,\,\,\,\,\,weakly-* in }}
\def\inn{\text{\,\,\,\,\,\,in }}
\def\cof{\mathop{\rm cof\,}\nolimits}
\def\Dn{\frac{\partial}{\partial N}}
\def\Dnn#1{\frac{\partial #1}{\partial N}}
\def\tdb{\tilde{b}}
\def\tda{b}
\def\qqq{u}
\def\lat{\Delta_2}
\def\biglinem{\vskip0.5truecm\par==========================\par\vskip0.5truecm}
\def\inon#1{\hbox{\ \ \ \ \ \ \ }\hbox{#1}}                
\def\onon#1{\inon{on~$#1$}}
\def\inin#1{\inon{in~$#1$}}
\def\FF{F}
\def\andand{\text{\indeq and\indeq}}
\def\ww{w(y)}
\def\startnewsection#1#2{\newpage \section{#1}\label{#2}\setcounter{equation}{0}}   
\def\nnewpage{ }
\def\sgn{\mathop{\rm sgn\,}\nolimits}    
\def\Tr{\mathop{\rm Tr}\nolimits}    
\def\div{\mathop{\rm div}\nolimits}
\def\curl{\mathop{\rm curl}\nolimits}
\def\dist{\mathop{\rm dist}\nolimits}  
\def\supp{\mathop{\rm supp}\nolimits}
\def\indeq{\quad{}}           
\def\period{.}                       
\def\semicolon{\,;}                  
\def\colr{\color{red}}
\def\colrr{\color{black}}
\def\colb{\color{black}}
\def\coly{\color{lightgray}}
\definecolor{colorgggg}{rgb}{0.1,0.5,0.3}
\definecolor{colorllll}{rgb}{0.0,0.7,0.0}
\definecolor{colorhhhh}{rgb}{0.3,0.75,0.4}
\definecolor{colorpppp}{rgb}{0.7,0.0,0.2}
\definecolor{coloroooo}{rgb}{0.45,0.0,0.0}
\definecolor{colorqqqq}{rgb}{0.1,0.7,0}
\def\colg{\color{colorgggg}}
\def\collg{\color{colorllll}}
\def\coleo{\color{colorpppp}}
\def\cole{\color{black}}
\def\colu{\color{blue}}
\def\colc{\color{colorhhhh}}
\def\colW{\colb}   
\definecolor{coloraaaa}{rgb}{0.6,0.6,0.6}
\def\colw{\color{coloraaaa}}
\def\comma{ {\rm ,\qquad{}} }            
\def\commaone{ {\rm ,\quad{}} }          
\def\cmi#1{{\color{red}IK: \hbox{\bf ~#1~}}}
\def\nts#1{{\color{red}\hbox{\bf ~#1~}}} 
\def\ntsf#1{\footnote{\color{colorgggg}\hbox{#1}}} 
\def\blackdot{{\color{red}{\hskip-.0truecm\rule[-1mm]{4mm}{4mm}\hskip.2truecm}}\hskip-.3truecm}
\def\bluedot{{\color{blue}{\hskip-.0truecm\rule[-1mm]{4mm}{4mm}\hskip.2truecm}}\hskip-.3truecm}
\def\purpledot{{\color{colorpppp}{\hskip-.0truecm\rule[-1mm]{4mm}{4mm}\hskip.2truecm}}\hskip-.3truecm}
\def\greendot{{\color{colorgggg}{\hskip-.0truecm\rule[-1mm]{4mm}{4mm}\hskip.2truecm}}\hskip-.3truecm}
\def\cyandot{{\color{cyan}{\hskip-.0truecm\rule[-1mm]{4mm}{4mm}\hskip.2truecm}}\hskip-.3truecm}
\def\reddot{{\color{red}{\hskip-.0truecm\rule[-1mm]{4mm}{4mm}\hskip.2truecm}}\hskip-.3truecm}
\def\tdot{{\color{green}{\hskip-.0truecm\rule[-.5mm]{6mm}{3mm}\hskip.2truecm}}\hskip-.1truecm}
\def\gdot{\greendot}
\def\bdot{\bluedot}
\def\pdot{\purpledot}
\def\ydot{\cyandot}
\def\rdot{\cyandot}
\def\fractext#1#2{{#1}/{#2}}
\def\ii{\hat\imath}
\def\fei#1{\textcolor{blue}{#1}}
\def\vlad#1{\textcolor{cyan}{#1}}
\def\igor#1{\text{{\textcolor{colorqqqq}{#1}}}}
\def\igorf#1{\footnote{\text{{\textcolor{colorqqqq}{#1}}}}}
\def\Sl{S_{\text l}}
\def\Sll{\bar S_{\text l}}
\def\Sh{S_{\text h}}
\newcommand{\p}{\partial}
\newcommand{\os}{{\overline{S}}}
\newcommand{\oos}{{\overline{\overline{S}}}}
\newcommand{\low}{\mathrm{l}}
\newcommand{\high}{\mathrm{h}}
\newcommand{\UE}{U^{\rm E}}
\newcommand{\PE}{P^{\rm E}}
\newcommand{\KP}{K_{\rm P}}
\newcommand{\uNS}{u^{\rm NS}}
\newcommand{\vNS}{v^{\rm NS}}
\newcommand{\pNS}{p^{\rm NS}}
\newcommand{\omegaNS}{\omega^{\rm NS}}
\newcommand{\uE}{u^{\rm E}}
\newcommand{\vE}{v^{\rm E}}
\newcommand{\pE}{p^{\rm E}}
\newcommand{\omegaE}{\omega^{\rm E}}
\newcommand{\ua}{u_{\rm   a}}
\newcommand{\va}{v_{\rm   a}}
\newcommand{\omegaa}{\omega_{\rm   a}}
\newcommand{\ue}{u_{\rm   e}}
\newcommand{\ve}{v_{\rm   e}}
\newcommand{\omegae}{\omega_{\rm e}}
\newcommand{\omegaeic}{\omega_{{\rm e}0}}
\newcommand{\ueic}{u_{{\rm   e}0}}
\newcommand{\veic}{v_{{\rm   e}0}}
\newcommand{\vp}{v^{\rm P}}
\newcommand{\tup}{{\tilde u}^{\rm P}}
\newcommand{\bvp}{{\bar v}^{\rm P}}
\newcommand{\omegap}{\omega^{\rm P}}
\newcommand{\tomegap}{\tilde \omega^{\rm P}}
\renewcommand{\vp}{v^{\rm P}}
\renewcommand{\omegap}{\Omega^{\rm P}}
\renewcommand{\tomegap}{\omega^{\rm P}}
\renewcommand{\d}{\mathrm{d}}
\newcommand{\R}{\mathbb{R}}
\newcommand{\C}{\mathbb{C}}
\newcommand{\N}{\mathbb{N}}
\renewcommand{\tt}{{\widetilde{\tau}}}
\newcommand{\im}{\mathrm{Im}}
\newcommand{\re}{\mathrm{Re}}
\newcommand{\id}{\mathrm{id}}

\newcommand{\tX}{{\widetilde{X}}}
\newcommand{\tY}{{\widetilde{Y}}}
\newcommand{\oY}{{\overline{Y}}}
\newcommand{\pmf}{{\psi^{(m)}_{\rm f}}}
\renewcommand{\pmb}{{\psi^{(m)}_{\rm b}}}

\begin{abstract}
We consider the three-dimensional ideal MHD system on a domain $\Omega' \subset \R^3$ with a part $\Gamma$ of the boundary~$\p \Omega$, where we prescribe both $u\cdot \nn$ and $b\cdot \nn$, while $u\cdot \nn = b\cdot \nn =0$ on $\p \Omega' \setminus \Gamma$. We prove the boundary controllability of the system, namely that we can prescribe the boundary data such that the unique solution of the system with initial state $(u_0,b_0)$ achieves another state $(u_1,b_1)$ in finite time, where $u_0,b_0,u_1,b_1$ are arbitrary divergence-free vector fields satisfying impermeability boundary condition which are extendable to vector fields with the same properties on any bounded domain obtained by extension of $\Omega'$ via $\Gamma$. As a byproduct, we give the first local well-posedness proof of incompressible, ideal MHD system, which does not use Elsasser variables and is thus applicable to any bounded domain with sufficient Sobolev regularity. We also provide a new and simple proof of the $2$D controllability.
\end{abstract}
\colb
\maketitle
\setcounter{tocdepth}{2}

\section{Introduction}\label{sec_intro}
Let $\Omega'\subset \R^3$ be a  bounded $H^{r+2}$ domain ($r\geq 3$), and let $\Gamma\subset \p \Omega$. We are concerned with the boundary controllability for the ideal magnetohydrodynamic (MHD) system on~$\Omega'$.  Namely, given $(v_0,b_0),(v_1,b_1)\in H^r$ satisfying the divergence-free condition, we look for $T>0$ and $k,l\in C([0,T];H^{r-1/2} (\Gamma ))$  such that there exists a unique solution $(v,b)$ to the system
\eqnb\label{mhd}
\begin{split}
\p_t v + v\cdot \nabla v - b\cdot \nabla b + \nabla p &=0,\\
\p_t b + v\cdot \nabla b - b\cdot \nabla v &=0,\\
\div \, v=\div\, b&=0 \quad \text{ in }  \Omega',\\
v\cdot \nn = b\cdot \nn &=0 \quad \text{ on } \p\Omega' \setminus \Gamma,\\
v\cdot \nn = k,\quad  b\cdot \nn &=l \quad \text{ on } \Gamma,
\end{split}
\eqne
where $\nn$ denotes the unit outward normal vector to $\p \Omega'$, which takes the state $(v_0,b_0)$ to $(v_1,b_1)$ in a given time~$T$.   The study of boundary controllability in inviscid models goes back to the work of Coron \cite{C1,C2,C3,C4}, who showed the boundary controllability for the $2$D incompressible Euler equation (see also the book \cite{C6}), and to the subsequent result of Glass \cite{G1,G2} regarding the $3$D incompressible Euler equation. These results answered a question posed in the paper of 
J.-L.~Lions~\cite{L3} on the relation between turbulence and controllability; see Section~7 therein and \cite[Section~4]{L1}.

We note that the ideal MHD system \eqref{mhd} is more challenging to treat than the incompressible Euler equations because (1) we have to deal with the interaction between two variables, $u$ and $b$, (2) pressure is more challenging to work with, particularly regarding the divergence-free constraint for $b$ (no pressure term in the equation for~$b$; see the discussion below), and (3) the local well-posedness of the perturbed MHD system on a general bounded domain is not clear due to the use of Elsasser variables (see the discussion below). 

We also note that while the results of Coron and Glass prove that the boundary controllability and null-controllability are equivalent in the case of the Euler equations, they are not for the MHD equations. In fact, as shown in \cite{KNV}, not every state is null-controllable, and, in the case of a rectangle, the paper \cite{KNV} provides a necessary and sufficient condition for being able to go from $(v_0,b_0)$ to~$(v_1,b_1)$.

The first result on controllability of the MHD system is due to  Rissel and Wang~\cite{RW1}, who showed the null-controllability for any state, however, by adding an additional small magnetic
force, i.e., by adding a forcing term in~\eqref{mhd}$_2$.
The difficulty in obtaining the null-controllability stems from the fact that the equation for the magnetic field does not contain the pressure term, so maintaining the divergence-free constraint for $b$ is unclear.

 Initially, the problem was solved in \cite{KNV} for the rectangle  $\Omega' = [0,1]^2 \subset \R^2$ with the controlled part of the boundary $\Gamma \coloneqq \{ (x,y) \colon x\in \{ 0,1\}, \,y\in (0,1) \}$. In order to describe the main idea, we note that if
\eqnb\label{knv_cond}
\int_{[0,1]^2} b_0\cdot e_1  = \int_{[0,1]^2} b_1\cdot e_1 = 0,
\eqne 
then one can reduce the boundary controllability problem on $\Omega'$ into a null-controllability of the problem extended to an infinite channel $\Omega \coloneqq \{ (x,y) \colon x\in \R , y\in (0,1) \}$ by considering extensions of $u_0$, $b_0$ into $\Omega$ such that $b_0$ is compactly supported in~$\Omega$. In fact, such extension of $b_0$ is possible if and only if \eqref{knv_cond} holds. In order to see this, one can first observe that, by the divergence-free condition, $\div\,b_0=0$, there exists a stream function $\Psi_0$ such that $b_0 = (-\p_2,\p_1)\Psi_0$. Consequently, the particle trajectories of the vector field $b_0$ are level sets of~$\Psi_0$. Thus, if $b_0$ is extendible to a compactly supported, divergence-free vector field on $\Omega$, then the top and bottom parts of $\p\Omega'$ must be part of the same streamline of $b_0$, and so the values of $\Psi_0$ there must be equal. Hence, for all $x\in (0,1)$,
\[
0= \Psi_0 (x,1) - \Psi_0 (x,0) = \int_0^1 \left( \Psi_0 (x',1) - \Psi_0 (x',0) \right) \,\d x' = \int_{\Omega'} \p_2 \Psi_0 = - \int_{\Omega'} b_0\cdot e_1 ,
\]
which explains the necessity of the compatibility condition~\eqref{knv_cond}. 

We note that we can assume that $\| v_0 \|_{H^r}$ and $\| b_0 \|_{H^r}$ are arbitrarily small and that $T>0$ is arbitrarily large, due to a scaling property of the ideal MHD system (see Step~2 in Section~\ref{sec_3d}, for example). The strategy for the 2D controllability then follows from the following two observations. The first is that there exist many nontrivial 2D Euler flows on $\Omega$ such as shear flows 
\eqnb\label{shear}
U=h(t) e_1 \qquad \text{ for any }h\colon (0,\infty ) \to \R.
\eqne
Thus one can think of a scenario where one considers $v=u+\us$, where $h(t)$ is a smooth function such that $h(0)=0$, and $\| u_0\|_{H^r}$ is taken small so that particle trajectories of $u(t)+\us (t) $ remain arbitrarily close to the particle trajectories of~$\us$. The second observation is that the support of $b$ is transported by $u(t)+\us (t)$. This is clear from the second equation in \eqref{mhd}, which becomes a transport equation for $b$ on $\{ b=0 \}$. Using these two observations, it is clear that $\supp\, b(\cdot ,t )$ will be transported away from $\Omega'$ in finite time, which gives the desired null-controllability result of the magnetic field $b$, see Figure~\ref{fig_knv_sketch} for a sketch. The velocity field $v$ can be then brought to $0$ using the theory of exact boundary controllability of the 2D Euler equations~\cite{C1,C2,C3,C4}.\\

\includegraphics[width=0.9\textwidth]{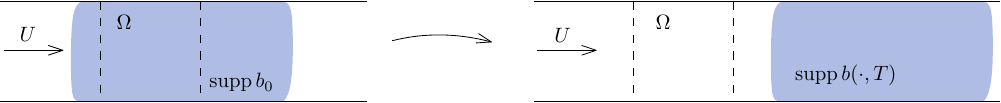}
  \captionof{figure}{A sketch of the shear flow method used by~\cite{KNV}.}\label{fig_knv_sketch} 

With this strategy in mind, the main difficulty obtaining the null-controllability is thus a local well-posedness result for~\eqref{mhd}, perturbed around a given 2D Euler flow~$U$. To this end, we write $(v,b)=(u+\us , b)$, and we use the Elsasser transformation,  
\eqnb\label{els}
z_1 \coloneqq u+b,\qquad z_2 \coloneqq u-b
.
\eqne
The ideal MHD system \eqref{mhd} for $u+\us$, $b$ then becomes
\eqnb\label{mhd_els}
\begin{split}
\p_t z_1 + (z_2 + \us )\cdot \nabla z_1 + \nabla q_1 &= 0,\\
\p_t z_2 + (z_1 + \us )\cdot \nabla z_2 + \nabla q_2 &= 0 ,\\
\div\, z_1=\div z_2 &=0,
\end{split}
\eqne
with appropriate boundary and initial conditions. To obtain the local well-posedness of the Elsasser system \eqref{mhd_els}, one can, for example, set $\omega_i \coloneqq \nabla^\perp \cdot z_i$, $i=1,2$ to obtain a ``vorticity-Elsasser-MHD'' system, from which $(z_1,z_2)$ can be recovered using De Rham's theorem. This is analogous to the vorticity  formulation of the Euler equations; see \cite[(2.11) and Lemma~2.4]{KNV} for details. We emphasize that the
use of Elsasser variables \eqref{els} is a well-known trick, which was originally introduced by Elsasser \cite{E}; however,  there is a difficulty in returning from the transformed system to the original one. In fact, given $z_1,z_2$ inverting the Elsasser transformation \eqref{els} only guarantees that $b$ satisfies the second equation in \eqref{mhd} with an additional pressure term $\nabla (q_1-q_2)/2$. Thus, in order to recover a solution $(u+\us,b)$ to the MHD system \eqref{mhd} one must show that $q_1=q_2$. This is possible in the case of the rectangle $\Omega'$ by noting that, at each time, $\p_2 (q_1-q_2)=0$ for $y\in \{ 0,1\}$, $\int_{\Omega'} (q_1 -q_2) =0$, and that, applying the divergence operator to the first two equations of \eqref{mhd_els} gives
\[
-\Delta (q_1-q_2)= \div  \left( ( z_2 +\us ) \cdot \nabla z_1 \right) - \div  \left( ( z_1 +\us ) \cdot \nabla z_2 \right) =\div  \left( \us  \cdot \nabla ( z_1 -z_2 ) \right) =(\na \us )^T \colon \nabla (z_1 -z_2) =0
\]
(see~\cite[p.~104]{KNV}), where we used the divergence-free condition for $u,b,\us$, and in the last equality; we also employed the fact that $\us$ is independent of $x,y$ (by \eqref{shear}). Thus $q_1=q_2$ by the classical elliptic theory, and consequently the desired local well-posedness of ideal MHD \eqref{mhd} follows from the Elsasser-MHD system~\eqref{mhd_els}. 

We emphasize that the fact that $\us$ is constant in space in the case of the rectangle $\Omega'$ is the main reason why the local well-posedness result is possible. This, together with the existence of the background flow $\us$, which guarantees the transport of $\supp\,b(\cdot ,t)$ away from $\Omega'$, are the main two reasons why the result of \cite{KNV} is possible. We note that these arguments cannot be generalized to an arbitrary 2D domain $\Omega'$ or the 3D case, resolved in this paper.
We note that the boundary controllability problem for a general 2D domain was recently resolved by Rissel~\cite{R}, who reduced the problem to interior controllability.

To state our main result, let $\Omega \subset \R^n$, $n=2,3$ be a bounded $H^{r+2}$  domain in the sense that there exists a Lipschitz function $d\colon \R^3 \to \R$ such that 
\eqnb\label{sobolev_domain}
d\in H^{r+2} (Q)\qquad \text{ and }\qquad  \p \Omega = \{ d=0 \}
\eqne
for some neighbourhood $Q$ of~$\p \Omega$. For example, if $\p \Omega$ is piecewise a graph of a $H^{r+2}$ function, then one can take $d$ as the signed distance function. Let $\Omega'\subset \Omega$ be a subdomain, set $\Gamma \coloneqq \p \Omega' \cap \Omega$, and let $u_0,b_0,u_1,b_1 \in \mathcal{H}_r$, where
\eqnb\label{def_Hr}
 \mathcal{H}_r \coloneqq \{ u\in H^r (\Omega ) \colon \div u =0\text{ in }\Omega , \, u\cdot \nn =0 \text{ on }\p \Omega  \},
\eqne
and $\nn$ denotes the outward normal vector to~$\p \Omega$. We note that we can assume that $b_0=b_1=0$ on $\Omega \setminus \overline{\Omega' + B(\delta )}$ for any preassigned $\delta $; see Corollary~\ref{cor_wlog} for details. 

\begin{theorem}[Main result: boundary controllability of the 3D and 2D ideal MHD system]\label{T01}
Let $r\geq 3$ in the case $n=3$ and $r\geq 2$ in the case $n=2$. There exist $T_0>0$ and $k,l\in C([0,T_0];H^{r-1/2} (\Gamma ))$ such that there exists a unique strong solution $(u,b)$ to \eqref{mhd} on time interval $[0,T_0]$ such that $(u(0),b(0))=(u_0,b_0)$ on $\Omega'$ and $(u(T_0 ),b(T_0 ))=(u_1,b_1)$ on~$\Omega'$.
\end{theorem}

We note that the above theorem covers both cases, $n=2$ and $n=3$, and that it is the first result on the existence or controllability of ideal MHD system \eqref{mhd} which does not use the Elsasser variables~\eqref{els}. In particular, it provides the first local well-posedness result for the ideal MHS system which holds in an arbitrary bounded $H^{r+2}$ domain. We refer the reader to the local well-posedness lemma (Lemma~\ref{L01}) for details.

In the case of $n=2$ of Theorem~\ref{T01}, an analogue of the shear flow \eqref{shear} can be found using complex analytical methods (see Section~\ref{sec_2d}), but  the existence of such flows  appears to be a difficult open problem in the case of $n=3$. In such case we develop a new method  of, roughly speaking, eliminating the magnetic field $b$ piece-by-piece. Namely, we will divide $\supp\,b_0$ into $M$ pieces such that, for each piece has a corresponding background flow $\us_{a}$, where $a=x_m$ for some $x_m\in \Omega'$, $m=1,\ldots , M$, with the property that $b$ can be ``transported away'' from the piece into $\Omega \setminus \Omega'$, where we will perform a number of ``surgeries'' on it (see Step~7 in Section~\ref{sec_3d}) that guarantee, after bringing $b$ back (via a reversed background flow), that $b=0$ on the given piece; see Figure~\ref{fig_sketch} for a sketch. The process is repeated finitely many times until $b|_{\Omega'}=0$.

We emphasize that the construction of the background flows in Step~3 of Section~\ref{sec_3d} is not based on shear flows but is instead inspired by the return method developed by Coron \cite{C5,C6} and used by Glass \cite{G1,G2} in the context of controllability of the 3D incompressible Euler equations. We will thus refer to these as the \emph{return flows}. We also note that each of the return flows $\us_{a}$, for $m=1,\ldots, M$, will solve  the forced incompressible Euler equations in $\Omega$, 
\eqnb\label{pert_Euler}
\begin{split}\p_t \us_a + (\us_a \cdot \nabla ) \us_a + \nabla p_a &=f,\\
\mathrm{div}\, \us_a &=g
\end{split}
\eqne
in $\Omega \times (0,T)$, where $T>0$, with some $p_a \in C^\infty (\overline{\Omega }\times [0,T]; \R^3)$ such that $\us_a \cdot \nn =0$ on $\p \Omega$, $\supp\, \us_a \subset \overline{\Omega }\times [T_0/4,3T_0/4]$, $\supp\,f (\cdot,t),\supp\,g(\cdot,t)  \subset \Omega \setminus \Omega'$. This means that each $\us_a$ solves the 3D incompressible Euler equations in $\Omega'$, but we also need to be extremely careful with the dynamics in $\Omega\setminus \Omega'$.  The main challenge is to ensure that $\div\, b(\cdot ,t)=0$ for all times, which is related to the fact mentioned above that there is no pressure term guaranteeing the preservation of divergence in the evolution equation for $b$ in~\eqref{mhd}.  To describe the challenge, we note that we will consider a solution $(v,b)$ such that
\eqnb\label{v_form}
v(x,t)= \us (x,t) + u(x,t).
\eqne
Then $(v,b)$ satisfies the ideal MHD equations on $\Omega$ with no penetration boundary conditions and the same forcing as in \eqref{pert_Euler}
 if 
\eqnb
  \begin{split}
   \partial_{t} u + u\nabla u + \us \nabla u + u \nabla \us - b\nabla
   b + \nabla p &= 0,\\
      \partial_{t} b
    + u \nabla b
    - b \nabla u
    + \us \nabla b
    - b \nabla \us
    &=0,
    \\
    \nabla \cdot u
    = \nabla \cdot b &= 0
    \inon{in $\Omega$,}
    \\
    u\cdot\nn = b\cdot \nn &= 0
    \inon{on $\partial\Omega $}
   ,
  \end{split}
   \label{EQ07a}
  \eqne
with the initial conditions
  \begin{equation}
   (u(0),b(0))
   = (u_0,b_0) = (v_0-U(0),b_0).
   \label{EQ26}
  \end{equation}
Thus 
\eqnb\label{div_b}
\p_t (\div \,b) + v\cdot \nabla (\div \,b) - b\cdot \nabla (\div \, \us ) =0\qquad \text{ in }\Omega ,
\eqne
which shows that $\div\,b=0$ is preserved, except when $\supp\,b(\cdot ,t ) \cap \supp \, (\div \, \us (\cdot ,t )) \ne \emptyset$. This means that we need to find a way to guarantee that the support of $b$ never touches $\supp \, (\div \, \us (\cdot ,t ))$, which is a major difficulty of this work and is addressed in Step~7 of Section~\ref{sec_3d} via a ``tentacle-cutting procedure'', i.e., a surgery on $b$ in $\Omega \setminus \Omega'$ that makes $b$ vanish near $\supp \, (\div \, \us (\cdot ,t ))$ whenever $\supp\,b(\cdot ,t)$ approaches it (i.e., whenever $\supp\, b$ pokes a tentacle where it is not welcome).\\

Having found a family of return flows $\{ U_{x_m} \}$, for $x_m\in \Omega'$, with $m=1,\ldots , M$, and being able to take care of the constraint $\div \, b=0$, the main difficulty now becomes to prove a local well-posedness result for~\eqref{EQ07a}. To this end, we develop a new method of constructing solutions using analytic spaces.

We define the analytic spaces in Section~\ref{sec_prelim_an} below and give a detailed proof of local well-posedness in Section~\ref{sec_lwp}. However, we point out here that the method yields a Sobolev solution $(u(t),b(t))$, which is defined on a domain $\Omega$ of merely Sobolev regularity. To achieve this, we first suppose that  $\Omega$ is an analytic domain and approximate the initial conditions by analytic functions.  For such approximations we use analytic estimates \eqref{an1}--\eqref{an3}, the product estimate \eqref{product_curved} and the pressure estimates \eqref{EQ27_curved}--\eqref{EQ27_acurved} to obtain local-in-time solutions satisfying both the Sobolev a~priori bound and an analytic a~priori bound (see~\eqref{001} and \eqref{002}, respectively). An important feature of such solutions is the \emph{persistence of analyticity}, namely that the analytic norms remain under control as long as the Sobolev norms do; see~\eqref{002c} for details. We then take the limit of the approximations as well as approximate a given Sobolev domain (recall~\eqref{sobolev_domain}) by an analytic one to obtain a local well-posedness result. This use of analytic spaces thus makes, generally speaking, for a robust existence and uniqueness theorem in Sobolev settings. In particular, thanks to this approach, we can completely avoid the use of Elsasser variables~\eqref{els}.\\

As for the proof of Theorem~\ref{T01} in the 2D case, we rely on a geometric proposition (Proposition~\ref{prop_conformal}), which allows us to find a bijection between an extension of any domain and a periodic  flat channel. To this end, we use the Riemann Mapping Theorem (Theorem~\ref{thm_riemann}); however, the main challenge is to guarantee that, from among all conformal mappings arising from the Riemann Mapping Theorem, we can find one guaranteeing periodicity of the mapping in the horizontal direction. In Section~\ref{sec_conform}, we prove, using geometric arguments, that this can be achieved.\\ 

Using the geometric proposition, we construct a background flow $\us$ which guarantees to bring $b$ outside, and then the proof of Theorem~\ref{T01} in the 2D case can be obtained following the same local well-posedness result (Lemma~\ref{L01}) as in the 3D case; see Section~\ref{sec_2d}.\\

The structure of the paper is as follows. In the following section, we introduce certain basic tools as well as discuss some techniques involving analytic spaces. In Section~\ref{sec_3d}, we prove Theorem~\ref{T01} in the 3D case, by describing seven steps leading to the proof. The steps use an algorithm for finding background flows (Section~\ref{sec_uS}), the local well-posedness lemma (Section~\ref{sec_lwp}), and a construction of a certain cancellation operator, which lets us shrink the support of the magnetic field $b$ (Section~\ref{sec_cancel}). Finally, we prove Theorem~\ref{T01} in the 2D case in Section~\ref{sec_2d}.

\section{Preliminaries}\label{sec_prelims}

First, we introduce some basic tools and then discuss analytic spaces in Section~\ref{sec_prelim_an} below.

For two open subsets $V,W\subset \Omega'$, we will write $V\subset \subset W$ to mean that $(V+B(\delta ))\cap \Omega' \subset W$ for some $\delta >0$.

We will use the following ODE fact:
\eqnb\label{ode_fact}
f'(t) \leq c H(t) f(t) + g(t)\qquad \text{ implies }\qquad f(t) \leq \left( f(0) + \int_0^t g(s) \d s \right) \mathrm{e}^{c\int_0^t H(s) \d s}
.
\eqne

Given a vector field $\us (x,t)$, we denote by $\Phi^{\us} (x,t,s)$ the particle trajectory along $\us$ at a time $t$, with the starting time $s$, i.e., the solution to
\eqnb\label{part_traj}
\begin{cases}
&\p_t \Phi^{\us} (x,t,s) = \us (\Phi^{\us} (x,t,s) ,t), \\
&\Phi^{\us} (x,s,s) =x,
\end{cases}
\eqne
and we set $\Phi^{\us} (x,t)\coloneqq \Phi^{\us} (x,0,t)$.\\

We also recall a version of a result of Cheng and Shkoller~\cite{CS} regarding solvability of a div-curl system on a bounded Sobolev domain. 

\begin{lemma}[Elliptic system involving div and curl]\label{L_CS}
Let $r>3/2$ and assume that $Q$ is a bounded $H^{r+2}$ domain. Let $f,g\in H^{l-1} (Q)$, where $l\in \{ 1, \ldots , r+1\}$, be such that $\div \, f =0$,  and $h\in H^{l-1/2}(\p Q)$ such that $\int_{\p Q} h\cdot \nn= \int_Q \div g$. Then there exists a unique solution $\phi \in H^l (Q)$ to
\[
\begin{cases}
\curl\, \phi  =f \qquad &\text{ in }Q,\\
\div \, \psi =g   &\text{ in } Q ,\\
\psi \cdot \nn =h &\text{ on }\p Q
\end{cases}
\]
and
\[
\| \phi \|_{H^l (Q)} \lec_{Q,l} \| f \|_{H^{l-1}(Q)} + \| g \|_{H^{l-1}(Q)} + \| h \|_{H^{l-1/2}(\p Q)}  .
\]
Moreover, if $h=0$ and $ f\cdot \nn =0 $ on $\p  Q$, then there exists a unique solution $\phi \in H^l (Q)$ to
\[
\begin{cases}
\curl\, \phi  =f \qquad &\text{ in }Q,\\
\div \, \psi =g   &\text{ in } Q ,\\
\psi \times \nn =0 &\text{ on }\p Q
\end{cases}
\]
such that $\| \phi \|_{H^l (Q)} \lec_{Q,l} \| f \|_{H^{l-1}(Q)} + \| g \|_{H^{l-1}(Q)} $.
\end{lemma}

\begin{proof}
See \cite[Theorem~1.1]{CS}.
\end{proof}

\subsection{Analytic spaces}\label{sec_prelim_an}

We say that a domain $\Omega$ is analytic if there exists a Lipschitz function $d\colon \R^3 \to \R$ such that $ \p \wo = \{ d =0 \}$ and
\eqnb\label{def_an_domain}
d \text{ is a real analytic function on }Q,
 \eqne
 where $Q$ is an open neighbourhood of~$\p \wo$. 
 
We introduce the Komatsu convention \cite{K1} on derivative notation,
\eqnb\label{komatsu_not}
\p^i \coloneqq \bigsqcup_{\alpha \in \{ 1,2\}^i} D^\alpha u,
\eqne
where 
\[
D^\alpha u \coloneqq \p_{\alpha_1} \ldots  \p_{\alpha_i} u.
\]
For example, $\p^2$ is the $2\times 2$ matrix of second order derivatives. Note that the same mixed derivative is included in $\p^i$ multiple times. 
As a consequence of such notation, we obtain
\[
\| D^i u \| = \sum_{\alpha \in \{ 1,2\}^i} \| D^\alpha u \|.
\]
Moreover, we have the product rule 
\eqnb\label{komatsu_product}
\p^i (fg) - f \p^i g = \sum_{k=1}^i {i \choose k } (\p^k f ) (\p^{i-k} g ),
\eqne
where the last product denotes the tensor product of derivatives that is consistent with \eqref{komatsu_not}
(see (5.1.3) in~\cite{K1}). The product rule \eqref{komatsu_product} is a consequence of the  identity
\[
\sum_{\substack{\alpha' \subset \alpha \\ |\alpha' |=k }} {\alpha \choose \alpha' } = {m \choose k},
\]
for each $m\in \N$, $k\in \{ 0 , \ldots , m \}$, and~$|\alpha |=m$. 
 
 We say that a vector field $T=\sum_{i=1}^n a_i \p_i$ is a \emph{tangential operator} to $\p \Omega$ if $T$ is a global analytic vector field such that $T\delta =0$ on $\p \Omega$, where $\delta (x) $ is the distance function to the boundary $\p \Omega$, taking positive values inside $\Omega$ and negative outside. We refer the reader to \cite[Section~2.1]{CKV} and \cite{JKL2} for an extensive discussions on tangential operators, but we note here that there exists a tensor $T^{j}$, using the Komatsu notation \eqref{komatsu_not}, such that its components span the tangent space at $x$ for each $x\in \p \Omega $. For example, as shown in \cite{JKL2}, we can use one tangential derivative in 2D and three in 3D, with a constant number $n(n-1)/2$ in any dimension~$n\geq 2$.
 
 We define analytic spaces as:
 \begin{align}
  \begin{split}
   &
   \Vert \psi \Vert_{X(\tau)}
   \coloneqq 
  \sum_{i+j\geq r } c_{i,j} 
    \Vert \p^i T^j  \psi \Vert , \qquad \text{ where } c_{i,j} \coloneqq \frac{(i+j)^r}{(i+j)!}\tau^{i+j-r }
  \overline{\epsilon}^i \epsilon^j,
    \\&
   \Vert \psi \Vert_{\tX(\tau)}
   \coloneqq 
   \Vert \psi \Vert_{X(\tau)}
   + \Vert \psi \Vert_{H^{r}}
  \\&
   \Vert \psi \Vert_{Y(\tau)}
   \coloneqq 
   \sum_{ i+j\geq r+1}
    \frac{
     (i+j)^{r+1}
        }{
     (i+j)!
    }
    \tau^{i+j-r-1}
    \overline{\epsilon}^i \epsilon^{j}
    \Vert \p^i T^j  \psi \Vert
    \\&
   \Vert \psi \Vert_{\widetilde{Y}(\tau)}
   \coloneqq 
  \tau  \Vert \psi \Vert_{Y(\tau)}
   + \Vert \psi \Vert_{H^{r}} \\&
   \Vert \psi \Vert_{\oY (\tau)}
   \coloneqq 
  \Vert \psi \Vert_{Y(\tau)}
   + \Vert \psi \Vert_{H^{r}},
  \end{split}
   \label{EQ19c}
  \end{align}
where all norms are taken on $\Omega$, the parameter $r\geq 3$ is fixed, and we use the notation
\[
\| \cdot \| \coloneqq \| \cdot \|_{L^2 (\wo )}
\]
for the $L^2$ norm.  It is well-known (see \cite{KP2}) that 
\eqnb\label{n_is_ana}
\nn \in \tX ( \tau_0; \widetilde{\Omega }),
\eqne
for some $\varepsilon_0>0$ and some neighbourhood $\widetilde{\Omega} \subset \Omega$ of $\partial \Omega$ in $\Omega$, and $\| f\|_{\tX (\tau ; \Omega')}\coloneqq \|  f\|_{X (\tau ; \Omega')}+ \| f \|_{H^r(\Omega')}$, where
\[
\| f\|_{X(\tau;\Omega')} \coloneqq \sum_{i+j\geq r} c_{i,j} \| \p^i T^j f \|_{L^2 (\Omega')}
\] 
denote the analytic spaces $\tX$ and $X$ on the subdomain $\Omega'$; see~\eqref{EQ19c} below.
In~\eqref{n_is_ana}, $\nn$ denotes the outward unit normal.

We note that  definitions \eqref{EQ19c} give 
\eqnb\label{an1}
\| f \|_{\tX} \lec_r \| f \|_{\tY },
\eqne
\eqnb\label{an2}
\| \na f \|_{\tX} \lec_r \| f \|_{\oY }
\eqne
and
\eqnb\label{an3}
\sum_{i+j\geq r} c_{i,j} \| \p^{i+1} T^j  f \|  \lec_r \| f \|_{Y }.
\eqne

\section{3D controllability}\label{sec_3d}

Here we prove Theorem~\ref{T01}. using the following steps.\\

\noindent\texttt{Step~1.} We first note that we can assume that $b_1=0$ and that $v_1$ is arbitrary. \\

\noindent Indeed, suppose that we have found $T_1 >0$ and $k,l\in C([0,T];H^{r-1/2} (\Gamma ))$ such that \eqref{mhd} has a unique solution $(v,b)$ on $[0,T]$ with $u(\cdot, 0 ) = u_0$, $b(\cdot ,0) = b_0$ such that $b(\cdot ,T)=0$. Thus, at the time $T$, the system becomes the 3D Euler equations, whose exact boundary controllability was established by Glass~\cite{G2}. Thus there exists $T'>0$ and $k\in C([T,T+T'];H^{r-1/2} (\Gamma ))$ such that the solution can be continued to $T+T'$ (where $l\coloneqq 0$ on $[T,T+T']$) to obtain $v(\cdot ,T+T')=b(\cdot ,T+T')=0$. Applying the same procedure to the initial data $(u_1,b_1)$ and using time-reversibility of the ideal MHD system, we can combine  such solutions and controls at $T+T'$ to obtain controls $k,l$ which drive an initial data $(u_0,b_0)$ to $(u_1,b_1)$ in a finite time. \\

\noindent\texttt{Step~2.} We  note that we can assume that $\| v_0 \|_{H^r} +\| b_0 \|_{H^r}  $ is arbitrarily small and that $T_0$ is arbitrarily large.\\

\noindent Indeed, given $v_0, b_0\in \mathcal{H}_r$ we consider the initial data $(\lambda v_0, \lambda b_0)$ for some small parameter $\lambda>0$. Suppose that we can solve the control problem with such initial data, i.e., obtain $T_0^{(\lambda )}>0$ and $k^{(\lambda )},l^{(\lambda )} \in C([0,T_0^{(\lambda )}];H^{r-1/2} (\Gamma )$ such that the unique solution $(v^{(\lambda )}, b^{(\lambda )})\in C([0,T_0^{(\lambda )}];H^r )$ to \eqref{mhd} with the initial data $(\lambda v_0, \lambda b_0)$ satisfies $b^{(\lambda )} (T_0^{(\lambda )})=0$. Then
\[
k(x,t)\coloneqq \frac{1}{\lambda} k^{(\lambda )} \left( x , \frac{t}{\lambda } \right), \qquad l(x,t)\coloneqq \frac{1}{\lambda} l^{(\lambda )} \left( x , \frac{t}{\lambda } \right)
\]
gives rise to the unique solution $v(x,t)\coloneqq \frac{1}{\lambda} v^{(\lambda )} \left( x , \frac{t}{\lambda } \right)$, $b(x,t)\coloneqq \frac{1}{\lambda} b^{(\lambda )} \left( x , \frac{t}{\lambda } \right)$ to \eqref{mhd} on the time interval of length $T_0^{(\lambda )}/\lambda$ with the initial data $(v_0,b_0)$.\vspace{0.5cm}\\

We now comment on the main strategy of the proof. For each $a\in \overline{\Omega'}$, we will  find an appropriate background flow $\us_a \in C([0,T_0]; H^{r+1} (\wo ))$ and we will study the incompressible ideal MHD system around $(\us ,0)$, posed on the extended domain $\Omega$, namely we will look looking for a solution $(u,b)$ of \eqref{EQ07a} on $\wo$ with $u_0, b_0 \in \mathcal{H}_r$ (recall~\eqref{def_Hr}). 

We emphasize that such reformulation gives us freedom in prescribing the dynamics of $u,b$ on $\wo \setminus \overline{\Omega '}$. This lets one use the \emph{return method} by first constructing $\us$ that, roughly speaking, ``flushes'' a piece of the fluid in $\Omega$ through $\wo\setminus \overline{\Omega '}$ over time interval $[0,T]$. Given the ``piece of the fluid'' which is, at a given time, traveling through $\wo\setminus \overline{\Omega '}$ we can modify it in an arbitrary way, so that, when it comes back to $\oo$ the magnetic field $b$ of the piece is zero. We will achieve this by choosing a sequence of time instances at which we will replace $b$ on a subset of $\wo \setminus \overline{\Omega '}$ by zero (for this, we will need an appropriate Bogovski\u{\i}-type lemma (see Step~6), since we need to retain the div-free condition and the slip boundary condition). Note that such strategy will work because the level set $\{ b=0 \}$ is transported by the velocity flow (recall~\eqref{mhd}), which  guarantees that the magnetic field $b$ will remain $0$, provided $\| u\|_{H^r}$ remains sufficiently small.\\

We now make this idea precise: In Step~3, we define an algorithm for generating the background flow  that can bring a given point $a\in \oo$ (and so also its neighbourhood $B(a,r_a)\cap \oo$) into $\wo\setminus \overline{\Omega }$ and then bring it back, by time reversal. In Step~4, we use these neighbourhoods to divide $\oo$ into pieces, and in Step~5 we provide a general local well-posedness lemma which allow us to solve the perturbed MHD system~\eqref{EQ07a} for $\psi(t)=(u(t),b(t))$ in $\wo$ around the return flows. In Step~6 we describe the cancellation operator which will let us modify $b$ in $\Omega \setminus \overline{\Omega'}$. Finally, we conclude the proof of Theorem~\ref{T01} in Step~7, where we first construct $\us$ by putting together the return flows corresponding to all the pieces, one after another, and then perform the modifications of the magnetic field $b$ at a discrete sequence of time instances when each piece is present in $\wo\setminus \oo$.\\

\noindent\texttt{Step~3.} (The return flow) Given $a\in \overline{\Omega'}$ and $T>0$ we construct a \emph{return flow}, namely a solution $\us_a \in C ([0,T];H^{r+1}({\wo })) $ to the forced Euler equations \eqref{pert_Euler} with some $p_a $ such that $\us_a \cdot \nn =0$ on $\p\Omega$, $\supp\, \us_a \subset \overline{\Omega }\times [T_0/4,3T_0/4]$, $\supp\,f (\cdot,t),\supp\,g(\cdot,t)  \subset \Omega \setminus \overline{\Omega'}$, and $\Phi^{\us_a} (a,T) \in \overline{\wo } \setminus \overline{\Omega'}$, where $\Phi^{\us}$ denotes the flow of $\us$; recall~\eqref{part_traj}.\\

\noindent The construction is inspired by the classical lemmas of Coron \cite[Lemma~A.2]{C3} and Glass \cite[Lemma~6.2]{G1}, and guarantees not only that $\Phi^{\us_a} (a,T) \in \overline{\wo } \setminus \overline{\Omega'}$, but also that $\Phi^{\us_a} (a,t) $ follows \emph{any prescribed path} starting from~$a$. In particular, the construction is based on the solvability of the problem of finding $\theta \in H^{r+2} (\Omega )$ satisfying
\eqnb\label{theta_problem}
\begin{cases}
\Delta \theta =0 \qquad &\text{ in }\Omega',\\
\p_{\nn } \theta =0 \qquad &\text{ on }\p \Omega' \setminus \Gamma,\\
\nabla \theta (\overline{x}) =v,
\end{cases}
\eqne
where $\overline{x}$ is an arbitrary point of $\overline{\Omega'}$, and, if $\overline{x}\in \Omega'$, $v$ is an arbitrary vector in $\R^3$ (if, otherwise, $\overline{x}\in \p \Omega' \setminus \Gamma$, then $v$ is an arbitrary element of the tangent space $T_{\overline{x}} (\p \Omega')$). Then the background flow $\us_a$ (for each $a\in  \overline{\Omega'}$) is constructed as 
\eqnb\label{temp00}
\us_a = \sum_{i=1}^l h_i (t) \nabla \theta_i (x)\in H^r (\Omega ),
\eqne
where $l\in \N$, the $h_i$'s are appropriate smooth cut-off functions in time, and the $\theta_i$'s are solutions to \eqref{theta_problem} for various choices of $\overline{x}$, $v$, determined by~$a$. \\

We refer the reader to Section~\ref{sec_uS} for details, but we emphasize at this point that the sum in \eqref{temp00} is finite, which allows us to approximate a Sobolev vector field $\us_a$, defined on $\Omega$, by another Sobolev vector field, defined on an analytic domain approximating $\Omega$, which is of the same structure as~\eqref{temp00}. It can also be made into an analytic vector field; see~Lemma~\ref{L_an_approx}. \\

\noindent\texttt{Step~4.} We construct a finite cover of $\supp\, b_0$ by open balls $B_m\coloneqq B(x_m,r_m)\subset \R^3$, where $r_m>0$, for $m\in \{ 1, \ldots , N\}$, such that $\Phi^{\us_{x_m}} (B(x_m,r_m)\cap \oo ,1) \subset \overline{\wo}\setminus \oo$.\\

To this end, we apply Step~3 with $T\coloneqq 1$  for each $x\in \oo$ to obtain the return flow $\us_{x}$ that takes $B(x,r_x)\cap \oo $ outside of $\oo$ in time $1$, where $r_x>0$. Then the claim follows from compactness of~$\oo$ and the fact that we can assume $b_0=0$ outside of $\bigcup_{m=1}^M B_m$; recall the comments before Theorem~\ref{T01}, as well as Corollary~\ref{cor_wlog}). \\

\noindent We note that Step~3 gives us return flows $\us_{x_m}$, for $m\in \{ 1, \ldots ,N \}$, such that $\Phi^{\us_{x_m}}(B_m\cap \oo ,1) \subset \overline{\wo}\setminus \oo$.  \\

\noindent\texttt{Step~5.} We prove a general local well-posedness of the MHD system on $\Omega$ near a background solution of the forced Euler equations~\eqref{pert_Euler}.\\

\noindent Namely, given $T>0$ and $\us \in C([0,T]; H^r (\Omega ))$, we  consider solutions $(v,b)$ of the form

  We will use the short-hand notation
  \eqnb\label{psi_notation}
  \psi(t) \coloneqq (u(t),b(t)),\qquad \psi_0 \coloneqq (u_0,b_0),
  \eqne
  and, in Section~\ref{sec_lwp}, we prove the following.
\cole

\begin{lemma}[local well-posedness of the perturbed MHD system]\label{L01}
Given $T>0$, a bounded $H^{r+2}$ domain $\Omega$, and $\us \in C([0,T]; H^{r+1} (\Omega ))$, there exists $\delta >0$ such that if $\| \psi_0 \|_{H^r} \leq \delta  $, then the system \eqref{EQ07a}--\eqref{EQ26} admits a unique solution $\psi \in C ([0,T];H^r (\Omega ) )$ satisfying the estimate 
\eqnb\label{a_priori_perturbedMHD}
\| \psi (t) \|_{H^r} \leq \| \psi_0 \|_{H^r} + C_r \int_0^t \| \psi (s) \|_{H^r} \left( \| \psi(s) \|_{H^r} + \| \us \|_{H^{r+1}} \right) \d s
,
\eqne
for all $t\in [0,T]$.
\end{lemma}\colb

We note that, since $\supp\, f , \supp\, g \subset \Omega \setminus \overline{\Omega'}$, the solution $(u,b)$ of \eqref{EQ07a} around $U$ provides us with a solution $(u+\us ,b)$ to the homogeneous ideal MHD \eqref{mhd} on $\Omega'$ with some nonhomogeneous boundary conditions $k,l\in C([0,T];H^{r-1/2} (\Gamma ))$.\\

\noindent\texttt{Step~6.} (the cancelling of $b$) Given a $H^{r+2}$ domain $\Omega $,   $H^{r+2}$ subdomains $\widetilde{W}, W\subset \Omega$  such that $\widetilde{W}\subset \subset W$, $\chi \in C^\infty (\Omega ; [0,1])$, and $b\in \mathcal{H}_r$ such that
\[
\text{ either } \chi (x) =1 \text{ or } b(x)=0 \quad \text{ for every } x\in W\setminus \widetilde{W},
\]
we construct $Tb \in \mathcal{H}_r$ such that
\eqnb\label{opT_prop}
\begin{cases}
\div\, Tb =0 \qquad &\text{ in }\Omega,\\
Tb =b &\text{ in } \Omega \setminus W,\\
Tb\cdot \nn =0 &\text{ on } \p \Omega ,\\
\|Tb\|_{H^r(\Omega) } \lec_{\Omega,W,\widetilde{W},\chi } \|b\|_{H^r(\Omega)}. &
\end{cases}
\eqne

We note that this step is reminiscent of a Bogovski\u{\i}-type cancellation: If one needs to consider a cutoff $\chi b$ of a  divergence-free vector field $b$ and find a correction of $\chi b$ which changes its divergence, $\div (\chi b)= b\cdot \na \chi$, into $0$ one can simply apply the Bogovski\u{\i} lemma (see~\cite{B1,B2} or \cite[Section~III.3]{Ga}) on~$W$. The difficulty of this step is to ensure that $b$ remains unchanged outside of $W$, and that the no-penetration boundary condition $b\cdot n$ is recovered. We construct operator $T$ in Section~\ref{sec_cancel}. \\

Having constructed the background flows $U_{x_m}$ in Step~4 and constructed solutions to the MHD system \eqref{EQ07a}--\eqref{EQ26} around them, we now set
\[
\psi_0 \coloneqq (u_0, b_0),
\]
where $u_0$ and $b_0$ denote the extensions to $\Omega$ such that $\supp\, b_0 \subset \bigcup_{m=1}^M B_m$. Now, we will construct a sequence of such solutions over consecutive time intervals such that the resulting glued solution (from one time interval to the next one) is continuous in time with values in $H^r (\Omega')$, with the first equal to $\psi_0$ and with $b=0$ at the final time.
We emphasize that the solutions are not going to be continuous in time with values in $H^r (\Omega )$, since we will perform certain cancellations of $b$ at (finitely many) times, which will be bounded operators on $H^r (\Omega)$, and will be localized in $\Omega \setminus \overline{\Omega'}$.\\

\noindent\texttt{Step~7.} (heuristic version) We construct forward and backward solutions and the cancelling operators.\\

Namely, for $m\in \{ 1,\ldots , M \}$ we will construct the forward solution $\psi_{\rm f}^{(m)} (\cdot , t)$ of the perturbed MHD system \eqref{EQ07a} around $U_{x_m}$ for $t\in (0,1]$, and, for $t\in (1,2]$, the backward solution $\psi_{\rm b}^{(m)} (\cdot , t)$, such that $\psi_{\rm b}^{(m)} (\cdot , 1)= \psi_{\rm f}^{(m)} (\cdot , t)$, and that $\Phi^{U_{x_m}} (B_m,1) \subset \Omega \setminus \overline{\Omega'}$. At time $t=1$, we will replace $b$ by $0$ in the most of $\Phi^{U_{x_m}} (B_m,1)$, except for a set $W_m$ on which we will apply the cancellation operation. The result is then brought back  by the background flow  $-U_{x_m}(\cdot ,2-t)$ over time interval $[1,2]$ via the backward solution $\psi_{\rm b}^{(m)}$; see~Figure~\ref{fig_sketch} for a sketch below. 

\includegraphics[width=\textwidth]{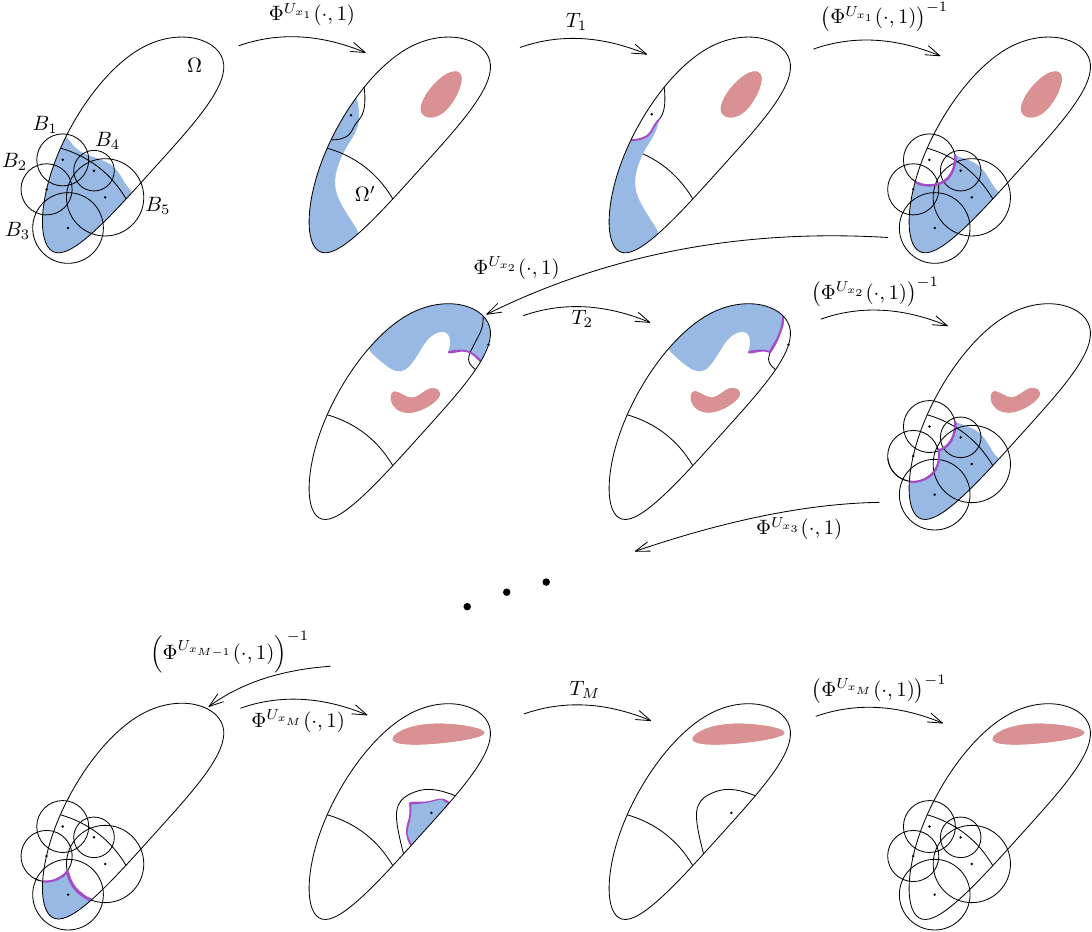}
  \captionof{figure}{A sketch of the forward and backward solutions and the cancelling operators. Here the purple sets in the sketches in the central columns correspond to the sets $W_m$, for $m=1,\ldots, M$, on which we perform the cancellation of~$b$. The red sets denote the supports of $\div\,U_{x_m}$.\vspace{0.5cm}}\label{fig_sketch} 

One difficulty is that the support of $b_{\rm f}^{(m)}(\cdot , t)$ could intersect with $\supp\, (\div U_{x_m} (\cdot ,t))$, which is depicted on Figure~\ref{fig_sketch} by the red region, for some $t\in (0,1]$. Whenever this happens, we will perform another cancelling operation on $b$ (see Step~7a below), which will eliminate a part of its support and give us an extra time to continue the evolution. There are three reason why such a cancellation resolves this issue: \begin{enumerate}
\item Since $\supp\, (\div U_{x_m} (\cdot ,t))\subset \Omega \setminus \overline{\Omega'}$, the cancellation takes place outside of $\Omega'$, which does not affect continuity in time of the solution $\psi$ inside the controlled domain $\Omega'$;
\item eliminating a part of the support of $b$ is not a problem; in fact it only helps us in achieving our goal of null-controllability of $b$;
\item we can guarantee that we will only need to perform finitely many of such operations.
\end{enumerate}

Nevertheless, the practical aspect of such cancellation is somewhat technical and requires a construction that ensures that the norms of all cancellation operators are independent of $\psi$. Also, and the small error that we make at each cancellation (which involves a slight increase of the support of $b$) does not get in the way of the overall null-controllability procedure.

\noindent\texttt{Step~7.} We construct sets $G_0, G_1, \ldots , G_M \subset \Omega$ such that
\eqnb\label{gmhm}
G_m \subset \subset H_m \coloneqq \bigcup_{k=m+1}^M B_k
,
\eqne
for all $m=0,\ldots ,M$, with $G_0 \coloneqq \supp\, b_0$, and operators $T_m \in B(\mathcal{H}_r)$, for $m=1,\ldots , M$, such that $(T_m f)(x) = f(x)$ for $x\in \Omega'$, with the following property:
For sufficiently small $\| \psi_0 \|_{H^r}$, we can use Lemma~\ref{L01} to construct, for each $m=1,\ldots , M$, the forward and backward solutions $\pmf$ and $\pmb$ on the time intervals $[0,1]$ and $[1,2]$, respectively, to the MHD system \eqref{EQ07a}--\eqref{EQ26}, perturbed around 
\[
\us_{x_m}(\cdot ,t) \quad \text{ and } \quad -\us_{x_m}(\cdot ,2-t ),
\]
respectively, such that 
\eqnb\label{bb}
b_{\rm b}^{(m)} (x ,2) =0 \qquad \text{ for }x\not \in G_m
\eqne
where $\pmb = (u_{\rm b}^{(m)}, b_{\rm b }^{(m)})$,
\[
\pmb (\cdot ,1 ) = T_m (\pmf (\cdot , 1)),
\]
\eqnb\label{ID_pmf}
\pmf (\cdot ,0 ) \coloneqq \begin{cases}
\psi_0 &m=1,\\
\psi_{\rm b}^{(m-1)}(\cdot , 2)\qquad &m>1
\end{cases}\eqne
for every $m\in \{ 1, \ldots , M\}$,
and
\eqnb\label{af}
\| \pmf (\cdot ,1) \|_{H^r} \lec C_m \| \pmf (\cdot ,0) \|_{H^r}, 
\eqne
\eqnb\label{ab}
\| \pmb (\cdot ,2) \|_{H^r} \lec C_m \| \pmb (\cdot ,1) \|_{H^r},
\eqne
by Lemma~\ref{L01}, provided $\| \psi_0 \|_{H^r}$ is sufficiently small.

Then Theorem~\ref{T01} follows by setting
\[
\psi (\cdot ,t) \coloneqq \begin{cases}
\pmf (\cdot, t-(2m-2))\qquad &\text{ for }t\in [2m-2,2m-1],\\
\pmb (\cdot, t-(2m-1))\qquad &\text{ for }t\in [2m-1,2m],
\end{cases}
\]
\[
\us (\cdot ,t) \coloneqq \begin{cases}
\us_{x_m} (\cdot, t-(2m-2))\qquad &\text{ for }t\in [2m-2,2m-1],\\
-\us (\cdot, 2m- t )\qquad &\text{ for }t\in [2m-1,2m]
\end{cases}
\]
for $t\in (2m-2,2m]$ and $m\in \{ 1, \ldots , M \}$. Indeed, $\psi + (\us ,0) \in C([0,2M]; H^r (\Omega' ))$ is then a solution to the incompressible ideal MHD \eqref{mhd} with the boundary controls $(k,l) \coloneqq \left. \left( \psi + (\us ,0) \right)\right|_{\Gamma} \in C([0,T_0];H^{r-1/2} (\Gamma ))$ and $b(2M) =0$ by \eqref{bb} since $G_M=H_M=\emptyset$.\\

\noindent\texttt{Step~7a.} We construct the forward solutions~$\pmf$.\\

Namely, given $\pmf (\cdot , 0)$ sufficiently small in $H^r$ and $b_{\rm f}^{(m)} (x,0)=0$ for $x\not \in G_{m-1} $, we construct a solution to the perturbed ideal MHD system \eqref{EQ07a} around $U_{x_m}(\cdot ,t)$ for $t\in [0,1]$. 

It might seem that this step follows directly from the local well-posedness Lemma~\ref{L01}, but it is not clear why $\div \,b_{\rm f}^{(m)}(\cdot ,t)$ would remain $0$ for all $t\in [0,1]$ (recall~\eqref{div_b}), since the support of $b_{\rm f}^{(m)}$ could intersect $\supp\, U_{x_m}$ at some $t<1$. To overcome this problem, we introduce the notion of a ``tentacle-cutting procedure'', which ``removes'' a part of the support of $b_{\rm f}^{(m)}$ every time it approaches $\supp\, U_{x_m}$ (i.e., when $b$ ``pokes a tentacle at $U_{x_m}$''). As a consequence, the support of $b_{\rm f}^{(m)}$ will increase a slightly at each such time instance, but this we will  keep under control. \\

To be precise, we first set
\eqnb\label{Em_def}
E_m \subset \subset \Omega' \setminus \overline{\Omega } \text{ be such that }  \supp U_{x_m}(\cdot , t) \subset \subset E_m
    \comma t\in [0,1],
\eqne
and let $\delta_m \in (0,1)$ be sufficiently small so that 
\[
E_m' \coloneqq E_m + B(\delta_m )   \subset \Omega'\setminus \overline{\Omega }.
\]

We let the zeroth cutting time be $0$, i.e., we set $\tau_0 \coloneqq 0$ and
\[
F_m^{(0)} (0)  \coloneqq G_{m-1}. 
\]
Given $l\geq 1$ such that $\tau_0,\ldots , \tau_{l-1} \in [0,1)$, we consider $F_m^{(l-1)}(\tau_{l-1})$ (i.e., the result of the $(l-1)$-th cutting), and we assume that 
\[\left( \Phi^{U_{x_m}} (\cdot , \tau_{l-1} )\right)^{-1} (F_m^{(l-1)}(\tau_{l-1})) \subset \subset H_{m-1}.\]
We propagate $F_m^{(l-1)}(\tau_{l-1})$ via the velocity $U_{x_m}(\cdot , t)$ from time $\tau_{l-1}$, that is, we define
\[
F_m^{(l-1)} (t) \coloneqq \Phi^{U_{x_m}} \left( F_m^{(l-1)}(\tau_{l-1}) ,t , \tau_{l-1} \right)  
\]
for $t\in (\tau_{l-1},1]$, and we consider the first time when $F_m^{(l-1)} (t)$ approaches $E_m$, i.e., we set 
\[
\tau_{l}\coloneqq \inf \{ t\in [\tau_{l-1},1] \colon F_m^{(l-1)}(t) \cap E_m \ne \emptyset \}
\]
if the set on the right-hand side is nonempty. On the other hand, if it is empty then we let $F_m^{(l)}(t)$ evolve until $t=1$ and set  $\tau_{l}\coloneqq 1$.

Given $\delta_{m,l}\in (0,\delta_m/4)$, we set
\[
\mathcal{R}_{m,l} \coloneqq \left( \p E_m'+B(\delta_{m,l}  ) \right) \cap \Omega ,
\] 
and we let $\chi_{m,l} \in C_0^\infty (\overline{\Omega},[0,1])$ be such that $\chi_{m,l} =1$ on $E_m' $ and $\chi_{m,l} =0$ outside $E_m'\cup \mathcal{R}_m $. We fix $\delta_{m,l}$ sufficiently small so that there exists an open set $W_m^{(l)} \subset \mathcal{R}_{m,l}$ such that
\eqnb\label{wound1}
\p E_m' \cap F_m^{(l-1)}(\tau_{l}) \subset W_m^{(l)}
\eqne
(so that $W_m^{(l)}$ is a neighbourhood of this part of $\p E_m'$ (i.e., the ``wound'' from the tentacle cutting)),
\eqnb\label{wound_big}
F_m^{(l-1)}(\tau_{l}) \cap \supp\, \na \chi_{m,l} \subset W_m^{(l)}
\eqne
(so that the ``wound'' is big enough to guarantee the divergence free condition for $b$ after cutting; see the comments below)
and 
\eqnb\label{wound_small}
\left( \Phi^{U_{x_m}} (\cdot , \tau_{l+1} ) \right)^{-1} (F_m^{(l)} (\tau_{l} ))\subset  \subset H_{m-1} 
,
\eqne
where
\[
F_m^{(l)} (\tau_{l} ) \coloneqq (F_m^{(l-1)}(\tau_{l}) \setminus E_m' )\cup W_m^{(l)}
\]
(i.e., the wound $W_m^{(l)}$ is small enough so that, if we would flow it back to $t=0$, it would remain strictly inside the part of $\Omega$ in which we still did not cancel $b$); see Figure~\ref{fig_tent} below for a sketch.

\includegraphics[width=\textwidth]{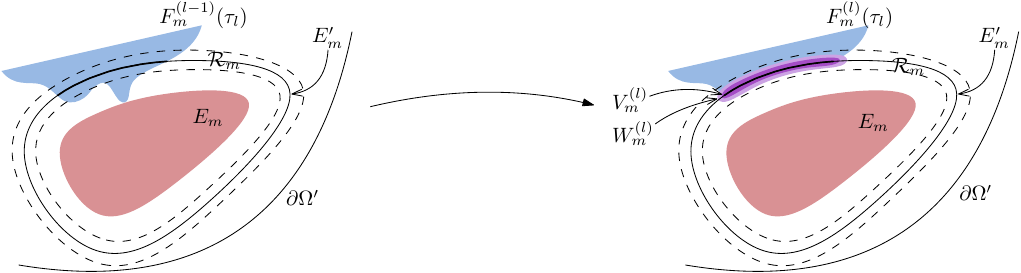}
  \captionof{figure}{A sketch of the tentacle-cutting procedure.}\label{fig_tent} 

We now let $V_m^{(l)}$ be any open subset of $ W_m^{(l)} $ such that  $V_m^{(l)} \subset \subset W_m^{(l)} $.
We now define the $l$-th tentacle-cutting operator $T_{m,l} \colon \mathcal{H}_r \to \mathcal{H}_r$ as
\[
b\mapsto T (b(1-\chi_{m,l} )),
\]
where $T$ is the cancelling operator of Step~6 applied with $V\coloneqq V_m^{(l)}$, $W\coloneqq W_m^{(l)}$, $Q\coloneqq \Omega$.

We note that the point of the tentacle-cutting operator is to make sure that the support of $b(t)$ never touches the support of $\div\,U_{x_m}$, so that the divergence-free property of $b$ is preserved. Indeed, if $b$ is any divergence-free vector field such that $b=0$ outside of $F_m^{(l-1)}(\tau_{l})$ then $T_{m,l} b$ is divergence-free  (by the construction of the cancelling operator $T$, recall Step 6).

We emphasize that the only requirement for smallness of $\delta_{m,l}$ comes from \eqref{wound_big}, since the smallness of $W_m^{(l)}$ (as required by \eqref{wound_small}) demands that $\chi_{m,l}$ goes from $1$ to $0$ over a very short distance. We emphasize that all the steps (i.e., the sets $E_m$, $E_m'$, $W_m^{(l)}$, $F_m$, the cutoff functions $\chi_{m,l}$, and the values of $\delta_m$, $\delta_{m,l}$, $\| T_{m,l} \|_{B(\mathcal{H}_r)}$) of the above tentacle-cutting procedure are independent of $\pmf$, as they depend only on $G_m$ and the return flow~$U_{x_m}$. In particular, the total number $L$ of the tentacle-cutting operators is bounded by 
\[
\frac{2 \max \{ \| U_{x_m}(\cdot, t)  \|_{L^\infty} \colon t\in [0,1] \} }{\delta_m} <\infty ,
\]
as $\delta_{m,l}< \delta_m/4$, so that, during time interval $(\tau_{l-1},\tau_{l})$, particles in $F_m^{(l-1)}(t)$ travel at least the distance $\delta_m/2$ with velocity at most $|U_{x_m} (t)|$. Let $F_m \subset \Omega$ be any open set such that
\eqnb\label{Fm_def}
F_m^{(L)}(1) \subset \subset F_m \quad \text{ and }\quad \left( \Phi^{U_{x_m}} (\cdot , 1) \right)^{-1} (F_m) \subset \subset H_{m-1}.
\eqne

We can now define $\pmf $ by applying the local well-posedness Lemma~\ref{L01} on each interval $[\tau_{l-1},\tau_{l}]$, where $l=1,\ldots , L+1$, and $\tau_{L+1}\coloneqq 1$. Namely, on $[\tau_{l-1},\tau_{l})$ we let $\pmf \coloneqq \psi_{\rm f}^{(m,l)}$, where $\psi_{\rm f}^{(m,l)}$ is the unique solution to the perturbed MHD system \eqref{EQ07a} around the return  flow $U_{x_m}(\cdot ,t)$ with the initial data
\[
\begin{cases}
\pmf (\cdot ,0) \qquad &l=1\qquad \text{(recall~\eqref{ID_pmf})},\\
\left( u_{\rm f}^{(m,l-1)}, T_{m,l-1}\left( b_{\rm f}^{(m,l-1)} (\cdot , \tau_{l-1})\right)\right) \qquad&l>1,
\end{cases}
\]
and we assume $\| \pmf (\cdot , 0 ) \|_{H^r}$ to be sufficiently small so that 
\[ \supp\, b_{\rm f}^{(m)}(\cdot ,t) \cap E_m =\emptyset\quad \text{ for all }t\in [0,1],
\]
and
\[
\supp\, b_{\rm f}^{(m)} (\cdot , 1) \subset F_m 
\]
which is possible by \eqref{Em_def} since taking $\| \pmf (\cdot , 0 ) \|_{H^r}$ small guarantees that $\| u_{\rm f}^{(m)} (\cdot ,t)\|_{H^r}$ remains small for all $t\in [0,1]$, and so the particle trajectories of the velocity $u_{\rm f}^{(m)}+U_{x_m} $ are close to the trajectories of~$U_{x_m}$.\\

\noindent\texttt{Step~7b.} We construct the backward solutions~$\pmb$.\\

Namely we construct an operator $T_m \in B(\mathcal{H}_r )$, for $m=1,\ldots ,M$, 
such that $T_m f = f$ on $\Omega'$, and a set $G_m \subset \Omega$ such that $G_m \subset \subset H_m$ (recall~\eqref{gmhm}) with the following property: Let $\pmb$ be the solution of the MHD system \eqref{EQ07a} on time interval $[1,2]$, perturbed around 
\[ \mathcal{V}_m (t)\coloneqq -U_{x_m} (\cdot ,2-t)\]
i.e., around the velocity field leading particles back to where they started at $t=0$, with the initial condition
\[
\pmb (\cdot , 1) \coloneqq  \left( u_{\rm f}^{(m)} (\cdot ,1) , T_m \left(  b_{\rm f}^{(m)} (\cdot ,1) \right)\right). 
\]
Then $b_{\rm b}^{(m)} (x ,2) =0$, for $x\in \Omega\setminus G_m$.

We now define~$T_m$. Given $\varepsilon_m \in (0,1)$, let 
\[
\mathcal{U}_m \coloneqq \left( \Phi^{U_{x_m}} (\p B_m ,1 )+B(\varepsilon_m ) \right) \cap \Omega'
\]
and let $\chi_m\in C_0^\infty (\overline{\Omega} , [0,1])$  be such that $\chi_m =1$ on $\Phi^{U_{x_m}} ( B_m ,1 )$ and $\chi_m =0$ outside $\Phi^{U_{x_m}} (\p B_m ,1 ) \cup \mathcal{U}_m$. We fix $\varepsilon_m>0$ sufficiently small so that there exists an open set $W_m \subset \mathcal{U}_m$ such that
\eqnb\label{cancelling_b}
\begin{split}
\Phi^{U_{x_m}} ( \p B_m ,1 ) \cap F_m &\subset W_m,\\
F_m \cap \supp\, \na \chi_{m,l} &\subset W_m,\\
 \Phi^{\mathcal{V}_m} (F_m',2, 1 ) & \subset  \subset H_{m},
\end{split}
\eqne
and
\eqnb\label{way_back}
\Phi^{\mathcal{V}_m} (F_m',t,1) \cap \supp\, \mathcal{V}_m (\cdot ,t ) =\emptyset \qquad \text{ for all }t\in [1,2],
\eqne
where
\[
F_m'\coloneqq (F_m \setminus \Phi^{U_{x_m}} (  B_m ,1 )) \cup W_m ;
\]
see Figure~\ref{fig_cancelling} for a sketch.

\includegraphics[width=\textwidth]{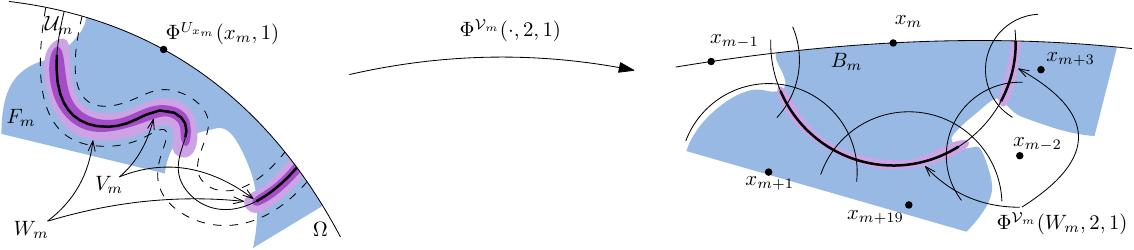}
  \captionof{figure}{A sketch of the cancelling procedure: The  grey set on the right-hand side is $\Phi^{\mathcal{V}_m} (F_m ,2,1)$, and the union of the blue set and the part of the grey set outside of $B_m$ is $\Phi^{\mathcal{V}_m} (F_m',2,1)$.}\label{fig_cancelling} 
  
The first three of the above properties are analogous to \eqref{wound1}--\eqref{wound_small}, except that now the last property is concerned with $H_m$ (rather than $H_{m-1}$). Indeed, the purpose of the current cancelling procedure is to extinguish $b$ on the most of $B_m$, i.e., to ensure that $b_{\rm b}^{(m)}(x , 2)\ne 0$ only for $x\in H_m$, rather than to ensure the support of $b_{\rm b}^{(m)}$ stays away from $\supp\,U_{x_m}$ (which is guaranteed by Step~7a). The property~\eqref{way_back} is necessary to guarantee that $\supp\, b_{\rm b }^{(m)}(\cdot ,t)$ (where $\pmb $ is defined below) will not touch $\supp\, \mathcal{V}_m(\cdot ,t)$ on the way back, i.e., for $t\in [1,2]$, so that $\div b_{\rm b }^{(m)}=0$ for all such $t$'s. 

We now let $V_m$ be any open subset of $W_m$ such that $V_m\subset \subset W_m$, and we define $T_m \colon \mathcal{H}_r \to \mathcal{H}_r$ as 
\[
b \mapsto T(b(1-\chi_m )),
\]
where $T$ is the cancelling operator of Step~6, applied with $V\coloneqq V_m$, $W\coloneqq W_m$, $Q\coloneqq \Omega$. 
We also pick an open set $G_m\subset \Omega$ such that
\[
 \Phi^{\mathcal{V}_m} (F_m',2, 1 ) \subset  \subset G_m \subset  \subset H_{m},
\]
which is possible by the last property in~\eqref{cancelling_b}.

As in Step~7a, we emphasize that the choice of $\varepsilon_m$, $\chi_m$, $W_m$, $V_m$ depends only on~$m$. We also see that 
$b_{\rm b}^{(m)} (\cdot , 1)$ is divergence free (by the second property in \eqref{cancelling_b}, as in Step~7a above), and, if $\| \pmf (\cdot , 1)\|_{H^r}$ is sufficiently small, then $b_{\rm b}^{(m)} (\cdot , t)$ remains divergence-free for all $t\in [1,2]$ (by~\eqref{way_back}) and $b_{\rm b}^{(m)} (x , t)\ne 0$ only for $x\in G_m$ (by the last property in~\eqref{cancelling_b}), as required. We note that, similarly to Step~7a above, this is possible, as taking $\| \pmf (\cdot , 1)\|_{H^r}$ small guarantees that $\| \pmb (\cdot , 1)\|_{H^r}$ is small (as $T_m\in B(\mathcal{H}_r)$), and so also $\| u_{\rm b}^{(m)} (\cdot ,t )\|_{H^r}$ remains small for all $t\in [0,1]$, which in turn guarantees that the particle trajectories of velocity $u_{\rm b}^{(m)} + \mathcal{V}_m$ remain close to the trajectories of~$\mathcal{V}_m$.

This concludes the construction of the forward and background solutions $\pmf$, $\pmb$, and so also concludes Step~7 and the proof of Theorem~\ref{T01} in the $3$D case.\\

In the remainder of this section, we prove the claims used in  the above proof.

\subsection{Construction of the background flow $u_S$}\label{sec_uS}

Here, for each $a\in \oo $, we construct the return flow $\us_a$ of Step~3, namely $\us_a\in C ([0,T] ; H^{r+1} (\Omega ))$, satisfying the forced Euler equation problem~\eqref{pert_Euler}. Namely, we set
\[
\us_a \coloneqq \nabla \theta,
\]
where $\theta $ is given by the following lemma.

\begin{lemma}[Glass flow]\label{L03}
For every $a\in \ooo$ there exists $\theta\in C ([0,T] ; H^{r+2} (\Omega ))$ such that $\mathrm{supp}\,\theta \subset \widetilde{\Omega } \times (0,1)$, $\theta=0$ for $t\in [0,1/4]\cup [3/4,1]$,
\[\begin{split}
\Delta \theta &=0 \qquad \text{ in } \overline{\Omega } \times [0,1],\\
\p_{\nn} \theta &=0 \qquad \text{ on } (\p \Omega \setminus \Gamma_0 )\times [0,1],\\
\Phi^{\nabla \theta } (a,1)&\in \overline{\widetilde{\Omega}} \setminus \overline{\Omega } .
\end{split}
\]
\end{lemma}
(Recall \eqref{part_traj} that $\Phi^{\nabla \theta}$ denotes the flow of~$\nabla \theta$.)
\begin{proof}
Suppose that $a\in \partial \oo$, let $b\in \partial  \wo \setminus \ooo$, and let $F\colon [0,1]\to \partial{\wo}$ denote a parameterization of a Sobolev smooth path connecting $a$ and $b$, i.e., $F(0)=a$, $F(1)=b$, $F'(t)\ne 0$ for all $t\in [0,1]$.
Note that
\eqnb\label{tangent_space}
\left\lbrace \nabla \theta (a) \colon \theta \in H^{r+2} (\wo ; \R ), \Delta \theta =0 \text{ in } \oo \text{ and } \p_{\nn} \theta =0 \text{ on } \p \oo\setminus \Gamma_0 \right\rbrace = T_a (\p \Omega )
;
\eqne
see~\cite[Lemma~6.2]{G1} for a proof. It follows from \eqref{tangent_space} that there exist $l\in \N$, $h_1,\ldots , h_l\in C ([0,1]; \R )$, $\theta_1,\ldots , \theta_l$ such that
\eqnb\label{choice1}
\begin{cases}
\Delta \theta_i =0 \qquad &\text{ in }\Omega \text{ and }\\
\p_{\nn} \theta_i =0  &\text{ on }\p \Omega\setminus \Gamma_0 
,
\end{cases}
\eqne
for all $i\in \{ 1, \ldots , l \}$, and there exists $\varepsilon >0$ such that
\eqnb\label{choice2}
\Phi^{\nabla \theta } (a,t) = F(t) \text{ for } t\in [0,t_0+\varepsilon ],
\eqne
where $t_0 \coloneqq \inf \{ t\in [0,1] : F(t) \not \in \overline\Omega \}$ is the first time the curve $F([0,1])$ leaves $\Omega$, and 
\[
\theta (x,t) \coloneqq \sum_{i=1}^l h_i (t) \theta_i (x).
\]

We emphasize that such choice ensures that the particle trajectory $\Phi^{\nabla \theta } (a,t)$ follows $F(t)$ \emph{exactly} (i.e., for all $t\in [0,t_0+\varepsilon ]$), even though we only apply \eqref{tangent_space} for finitely many $a$'s. To see that such choice exists, consider, for $M\in \N$, a sequence of points $\{ x_m \}_{m=0}^M\subset F([0,1])$ such that $x_0=a$, $x_M=b$,$x_{m_0} = F(t_0)$ (for some $m_0\in \{ 0,\ldots , M\}$) and 
\eqnb\label{def_partition}
(B(x_m,r )\cap \p \widetilde{\Omega} )\cap (B(x_{m+1},r )\cap \p \widetilde{\Omega} )\ne \emptyset 
,
\eqne
for $m=0,\ldots , M-1$, for some $r>0$.

Given $m\in \{ 0, \ldots , M\}$, let $v_1,v_2\in T_{x_m}(\p \widetilde{\Omega })$ be a basis of $T_{x_m}(\p \widetilde{\Omega })$ such that $v_1= F'(t_m)$, where $t_m\in [0,1]$ is such that $F(t_m)=x_m$. Let $\theta_{v_1}$ be given by \eqref{tangent_space}, and let $g_1 \in C_0^\infty ((0,1);\R)$ be such that $g_1 (t)=1$ for $t$ in some neighbourhood of~$t_m$. Then the particle trajectory $\Phi^{g_1 \nabla \theta_{v_1}}(x_m,t,t_m)$ passes through $x_m$ at $t=t_m$ and approximates the path $F(t)$ for $t$ close to~$t_m$. To make sure that the particle trajectory follows $F(t)$ \emph{exactly} we can thus replace the velocity field $g_1 \nabla \theta_{v_1} $ by $g_1(t) \nabla \theta_{v_1} + g_2(t) \nabla \theta_{v_2}$, where $g_2 (t)$ is chosen appropriately  so that $\Phi^{g_1(t) \nabla \theta_{v_1} + g_2(t) \nabla \theta_{v_2}} (x_m , t,t_m) = F(t)$ for all $t$ in a neighbourhood of~$t_m$. Clearly, choosing  $M\in \N$ sufficiently large, $r>0$ sufficiently small (so that $\{\Phi^{b\nabla \theta_{v_1} + c \nabla \theta_{v_2}}(x_m,t,t_m ) \colon b,c \in [-1,1], t\in [t_m-\delta , t_m+\delta ] \} \supset B(x_m,r)\cap \p \widetilde{\Omega} $ for all $m$, for some $\delta >0$), and ensuring that functions $h_i$'s involve appropriate partition of unity, we obtain \eqref{choice1}--\eqref{choice2}. We note that \eqref{choice2} holds since  $F'(x_{m_0})$ points towards $\p \widetilde{\Omega} \setminus \overline{\Omega }$, and so  $\Phi^{\nabla \theta} (a, t)\not \in \overline{\Omega }$ for $t\in (t_{m_0},t_{m_0}+\varepsilon)$, for some $\varepsilon >0$.\\

Clearly, the claim follows by rescaling the time (as $\phi^{\nabla \theta } (a,t_0+\varepsilon) \in \overline{\widetilde{\Omega}}\setminus \overline{\Omega }$), or, alternatively, by (smoothly) extending $\nabla \theta$ by $0$ for $t\in [t_0+\varepsilon ,1]$.

If $a\in \Omega$, the claim follows in an analogous by noting that, for such $a$, the left-hand side of \eqref{tangent_space} equals to $\R^3$; see \cite[Lemma~6.1]{G1}. 
\end{proof}

\subsection{Proof of the local-wellposedness lemma}\label{sec_lwp}

Here we prove Lemma~\ref{L01}. In order to discuss our strategy, we first note that the classical $H^r$ estimate for the perturbed MHD system \eqref{EQ07a},
\eqnb
  \begin{split}
   \partial_{t} u + u\nabla u + \us \nabla u + u \nabla \us - b\nabla
   b + \nabla p &= 0,\\
      \partial_{t} b
    + u \nabla b
    - b \nabla u
    + \us \nabla b
    - b \nabla \us
    &=0,
    \\
    \nabla \cdot u
    = \nabla \cdot b &= 0
    \inon{in $\Omega$,}
    \\
    u\cdot\nn = b\cdot \nn &= 0
    \inon{on $\partial\Omega $}
   ,
  \end{split}
   \label{EQ07a_repeat}
  \eqne
is
\eqnb\label{001}
\frac{\d }{\d t} \| \psi (t) \|_{H^r} \leq C_r \left(  \| \psi (t) \|_{H^r}^2 + \| \us  \|_{H^r} \| \psi (t) \|_{H^r} \right),
\eqne
provided all the steps in the derivation of the $H^r$ can be justified. Recall that we use the short-hand notation \eqref{psi_notation}, $\psi(t) = (u(t),b(t))$, for a solution of \eqref{EQ07a_repeat} and $\psi_0 \coloneqq (u_0,b_0)$ for the initial data. 

 Noting that
\eqnb\label{def_g}
g(t) \coloneqq  \| \psi_0 \|_{H^r} \exp \left( C_r t (1+\sup_{s \geq 0 } \| U(s) \|_{H^{r+1}} )\right)
\eqne
satisfies 
\eqnb\label{gprime}
g' \geq C_r g (g + \| \us  \|_{H^r} ) \comma t\leq T_0 \coloneqq \frac{- \log \| \psi_0 \|_{H^r }}{C_r (1+\sup_{s\geq 0} \| U(s) \|_{H^{r+1}})}
\eqne
with the same initial data, this suggests that 
\[
\| \psi (t) \|_{H^r} \leq g(t)
\comma t\leq T_0 ,
\]
which suggests that the local well-posedness result holds. However, as mentioned in the introduction, the construction of solutions is a major difficulty in the case of a Sobolev domain~$\Omega$. To overcome this challenge, we will first construct a solution in the case when $\Omega $ is an analytic domain.

 We also assume that $u_0, b_0$ and $\us $ are analytic in the sense that there exists $\tau_0>0$ such that
 \eqnb\label{def_of_an}
 u_0, b_0 \in \tX (\tau_0 ) \cap \oY (\tau_0 ) \qquad \text{ and } \us \in C([0,T];\tX (\tau_0 ) \cap \oY (\tau_0 )),
 \eqne
 where the analytic spaces $\tX, \oY$, and also $X,Y, \tY$ are discussed in Section~\ref{sec_prelim_an}. We prove in Section~\ref{sec_constr_anan} below that under such assumption there exists a unique solution $\psi \in C([0,T_0];\tX (\tau (t)))$, for some choice of decreasing $\tau \colon [0,T_0] \to (0,\infty )$, which also satisfies the a~priori estimate \eqref{001} for $t\in [0,T_0]$. In Section~\ref{sec_constr_soban}, we then obtain a unique solution $\psi \in C([0,T_0];H^r ) $ to \eqref{EQ07a_repeat} for Sobolev initial data $\psi_0$ and Sobolev background $U$, considered on an analytic domain~$\Omega$. Finally, in Section~\ref{sec_constr_sobsob}, we construct a unique Sobolev solution considered on a Sobolev domain~$\Omega $.

\subsubsection{Construction of analytic solution on an analytic domain}\label{sec_constr_anan}

Here we suppose that $\Omega$ is analytic in the sense of \eqref{def_an_domain} and, given initial data $\psi_0$ and background $\us$ satisfying \eqref{def_of_an}, we find an analyticity radius function $\tau \in C^1 ([0,T_0];(0,\infty))$, and a unique solution $\psi \in C^0 ([0,T_0];\widetilde{X}(\tau(t)) )$ to \eqref{EQ07a_repeat} satisfying the \emph{analytic a priori estimate}
\eqnb\label{002}
\frac{\d }{\d t } \| \psi \|_{\tX } - \dot \tau \| \psi \|_{Y} \leq C_r \left(     \| \psi \|_{\tX }\| \psi \|_{\tY}  + \| \psi \|_{\tX }  \| \us \|_{\oY} + \| \us \|_{\tX } \| \psi \|_{\tY}  \right)
,
\eqne
for all $t\in [0,T_0]$, where $T_0$ is given in~\eqref{gprime}.

We emphasize that, even though we aim to find an analytic solution $\psi$, the time $T_0$ of existence depends only on the Sobolev norm $\| \psi_0 \|_{H^r}$ of the initial data. This is possible thanks to \emph{persistence of analyticity}; see \eqref{002a} below. \\

We first derive the analytic a~priori bound~\eqref{002}. To this end we apply $\p^\alpha$ to the first two equations of  \eqref{EQ07a} and multiply by $(\p^\alpha u,\p^\alpha b)$ to obtain 
  \begin{align}
  \begin{split}
   &
   \frac12
   \frac{\d}{\d t}
   \int
   (|\partial^i T^j u|^2
    + |\partial^i T^j b|^2)
    \\&\indeq
    =-
         \int ( (u+\us)\cdot \nabla  \partial^i T^j u) \cdot \partial^i T^j u
      -
   \int ((u+\us)\cdot \nabla \partial^i T^j  b)\cdot  \partial^i T^j b
  \\&\indeq\indeq
  +
     \int
       \left( (b\cdot \nabla  \partial^i T^j b )\cdot \partial^i T^j u + (b\cdot \nabla  \partial^i T^j u ) \cdot \partial^i T^j b   \right) 
        - \int S_{ij} (u+\us  ,u )\cdot \partial^i T^j u
	\\&\indeq\indeq
 - \int S_{ij} (u+\us ,b )\cdot \partial^i T^j b   +\int \left( S_{ij} (b, b) \cdot \partial^i T^j u + S_{ij} (b,u)\cdot \partial^i T^j b \right) 
   \\&\indeq\indeq
    -\int \left(u \cdot \nabla \partial^i T^j \us +S_{ij} (u,\us ) \right)\cdot \partial^i T^j u + \int \left( b \cdot \nabla \partial^i T^j \us  + S_{ij} (b,\us ) \right) \cdot \partial^i T^j b
   \\&\indeq\indeq
  +
   \int \partial^i T^j\nabla p \cdot \partial^i T^j u
   ,
  \end{split}
   \label{EQ08}
  \end{align}
  where we set
\eqnb\label{def_Salpha}
S_{ij} (v,w) \coloneqq \p^i T^j ((v\cdot \nabla )w )-v \cdot \nabla \p^i T^j w .
\eqne  
To estimate the terms involving $S(\cdot , \cdot )$ in \eqref{EQ08}, we will need use the inequality 
\eqnb\label{product_curved}
\sum_{i+j\geq r }c_{i,j}   \| S_{i j} (u,v)\|  \lec_r \| v \|_{\tY (\tau )} \| u \|_{\tX (\tau )} +\| v \|_{\tX (\tau )} \| u \|_{\tY (\tau )}
;
\eqne
see \cite[(3.15)]{KOS} for a proof. In order to estimate the pressure term, we note that, applying the divergence operator to  the  evolution equation for $u$ (i.e., the first equation in \eqref{EQ07a}) gives
  \begin{align}
  \begin{split}
   -\Delta p
   &=
   \nabla \cdot (u\cdot \nabla u
                  + \us\cdot  \nabla u
                  + u\cdot \nabla \us		  
                  - b\cdot \nabla b )
    =
    \nabla u \odot \nabla u
    + 2 \nabla \us \odot \nabla u
     - \nabla b \odot \nabla b + u\cdot \nabla (\div \us)
   ,
  \end{split}
   \label{EQ13}
  \end{align}
where we set
  \begin{equation}
   \nabla u \odot \nabla v
   \coloneqq
    (\nabla u)^{T}:\nabla v
   =
   \partial_{i} u_j \partial_{j} v_i
   \llabel{EQ14}
  \end{equation}
and we also used the facts $\nabla \cdot (u\nabla v) = \nabla u\odot \nabla v + u \cdot \nabla (\nabla \cdot v)$
and $\nabla u \odot \nabla v= \nabla v \odot \nabla u$. As for the boundary condition for $p$, we apply the dot product with the normal~$\nn$ to the evolution equation for $u$, which gives 
  \begin{align}
  \begin{split}
   &
   - \nabla p \cdot \nn
   = (u\cdot \nabla u + \us\cdot \nabla u + u\cdot\nabla \us - b\cdot\nabla b)\cdot \nn
   .
  \end{split}
   \label{EQ15}  
  \end{align}
We will use the following pressure bound.

\begin{lemma}[The pressure estimate in an analytic norm]\label{L_pressure}
Consider the solution $p$ of the Neumann problem
  \begin{align}
  \begin{split}
   &\Delta p = \nabla f_1 : \nabla f_2
   \hspace{1cm}\inon{in $\Omega$}
   \\&
   {\p_{\nn}} p = ((g_1\cdot \nabla g_2)\cdot \mathsf{n})|_{\partial\Omega}
   \inon{on $\partial\Omega$}
   ,
  \end{split}
   \label{EQ25_curved}
  \end{align}
with the normalizing condition $\int_\Omega p= 0$,
where $g$ is defined in $\Omega$,
with the necessary compatibility condition. Then we have
  \begin{align}
  \begin{split}
   &
       \Vert \nabla p\Vert_{ X(\tau )}
   \lec
    \Vert f_1\Vert_{\tX (\tau)}
    \Vert f_2\Vert_{\tX (\tau)}
    +  \Vert g_1 \Vert_{\tX (\tau)}\Vert g_2 \Vert_{\oY (\tau)}
 ,
  \end{split}
   \label{EQ27_curved}
  \end{align}
provided ${\epsilon}$ and $\overline{\epsilon }/\epsilon$ are sufficiently small positive numbers. Moreover, if also $g_1\cdot \mathsf{n}= g_2 \cdot \mathsf{n}=0$ then 
 \begin{align}
  \begin{split}
   &
       \Vert \nabla p\Vert_{ X(\tau )}
   \lec
    \Vert f_1\Vert_{\tX (\tau)}
    \Vert f_2\Vert_{\tX (\tau)}
    +  \Vert g_1 \Vert_{\tX(\tau)}\Vert g_2 \Vert_{\tX(\tau)}
   .
  \end{split}
   \label{EQ27_acurved}
  \end{align}\colb
\end{lemma}  

\begin{proof}
See \cite[Lemma~3.3]{KOS}.
\end{proof}

Using Lemma~\ref{L_pressure}, we can estimate all components of the pressure function arising from the boundary value problem \eqref{EQ13}--\eqref{EQ15}, except for the term involving $\div\,\us$. To this end, we can use the following statement.

\begin{corollary}\label{cor_pressure}
Consider the solution $p$ of the Neumann problem
  \begin{align}
  \begin{split}
   &\Delta p = f_1 \cdot  \nabla f_2
   \hspace{1cm}\inon{in $\Omega$}
   \\&
   {\p_{\nn}} p = 0
   \inon{on $\partial\Omega$}
   ,
  \end{split}
   \label{EQ25_easier}
  \end{align}
with the normalizing condition $\int_\Omega p= 0$,
where $f_1,f_2$ satisfy $\int_{\p \Omega } (f_1 f_2) \cdot \nn =0$.  Then we have
  \begin{align}
  \begin{split}
   &
       \Vert \nabla p\Vert_{ X(\tau )}
   \lec
    \Vert f_1\Vert_{\tX  (\tau)}
    \Vert f_2\Vert_{\tX (\tau)},
  \end{split}
  \end{align}
provided ${\epsilon}$ and $\overline{\epsilon }/\epsilon$ are sufficiently small positive numbers. 
\end{corollary} 
  Applying Lemma~\ref{L_pressure} and Corollary~\ref{cor_pressure} to \eqref{EQ13}--\eqref{EQ15} gives that 
\[
\| \na p \|_{X} \lec \| \psi \|_{\tX}^2 + \| \psi \|_{\tX} \left( \| \us \|_{\tX} + \| \div \us \|_{\tX } \right) 
.
\]
Thus, noting that all term in the second and third lines of \eqref{EQ08} that do not involve $U$ vanish, due to the divergence-free condition on $u$ and $b$, we can sum \eqref{EQ08} in $i$, $j$ with coefficients $c_{i,j}$ to obtain
\[\begin{split}
\frac{\d }{\d t } \| \psi \|_{\widetilde{X}} - \dot \tau \| \psi \|_{Y} &\lec_r \| \psi \|_{\widetilde{X}} \left( \| \div\, \us \|_{L^\infty } + \| \us \|_{\tY} + \| \psi \|_{\tY }   \right)
\\&\indeq
+  \| \us \|_{\tX }  \| \psi \|_{\tY}
+ \| \psi \|_{L^\infty } \sum_{i+j\geq r } c_{i,j} \| \na \p^i T^j \us  \| + \| \na p \|_{\tX }\\
&\lec_r  \| \psi \|_{\tX }  \| \us \|_{\oY} + \left(  \| \psi \|_{\tX } +  \| \us \|_{\tX } \right)  \| \psi \|_{\tY}  ,
\end{split}
\]
which shows~\eqref{002}. We can now choose $\tau (t)$ such that \eqref{002} lets us control the solution until the final time $T=T(\| \psi \|_{H^r})$. Namely, if $\tau$ satisfies
\eqnb\label{tau_choice}
-\dot \tau \geq 2 C_r \tau (\| \psi \|_{\tX}  + \| U \|_{\tX} )
,
\eqne
then \eqref{002} gives
\eqnb\label{002a}
\frac{\d }{\d t } \| \psi \|_{\tX } - \frac{\dot \tau}2  \| \psi \|_{Y } \leq C_r \left( \| \psi \|_{\tX } ( \| \psi \|_{H^r}+   \| \us \|_{\oY } ) + \| \psi \|_{H^r }\| \us \|_{\tX } \right).
\eqne
This shows that the $L^\infty_t \tX \cap L^1_t \tY$ norm of $\psi$ remains under control at least as long as $\| \psi (t) \|_{H^r}$ remains bounded, provided the analyticity radius $\tau (t)$ satisfies~\eqref{tau_choice}. As for the $H^r$ estimate, we have
\[
\frac{\d }{\d t} \| \psi \|_{H^r} \le_r  \| \psi \|_{H^r}^2 + \| \psi \|_{H^r} \| U \|_{H^{r+1}},
\] 
by \eqref{EQ08}, which in particular implies that 
\eqnb\label{002b}
\| \psi (t) \|_{H^r} \leq  \| \psi_0 \|_{H^r} \exp \left( C_r t (1+\sup_{s \geq 0 } \| U(s) \|_{H^{r+1}} )\right)=: g(t)
\comma t\leq \frac{- \log \| \psi_0 \|_{H^r }}{C_r (1+\sup_{s\geq 0} \| U(s) \|_{H^{r+1}})}.
\eqne
For such $t$, the inequality \eqref{002a} gives
\eqnb\label{002c}
\| \psi(t) \|_{\tX } \leq \left( \| \psi_0 \|_{\tX}+\int_0^t g(s) \| U(s) \|_{\tX (\tau (s))} \d s \right) \exp \left( C_r \int_0^t \left( g(s) + \| U(s) \|_{\oY (\tau (s))}\right) \d s \right) =: G(t)
\eqne
(recall~\eqref{ode_fact}), provided \eqref{tau_choice} holds. We can thus choose $\| \psi_0 \|_{H^r}$ sufficiently small, so that the analytic solution exists for a given time. Moreover, we can now fix $\tau (t)$, by imposing the ODE of a similar form to \eqref{tau_choice}, except with $\| \psi \|_{\tX}$ replaced by $G$, i.e., we define $\tau (t)$ by the ODE 
\eqnb\label{tau_choice_actual}
\dot \tau (t) = - C_r \tau (t) \left( G(t) + \| U(t) \|_{\tX (\tau (t))} \right)=:f(t),\qquad \tau (0) = \tau_0.
\eqne
With such choice of $\tau$, the a~priori estimate \eqref{002a} holds unconditionally, and moreover a simple Picard iteration (see \cite[Section~2.2]{KOS} for details) lets one construct a solution $\psi \in C([0,T_0];\tX (\tau ))$ such that $-\dot \tau \psi \in L^1((0,T_0);\tY (\tau ))$ and the upper bounds \eqref{002b}, \eqref{002c} hold.\\

\subsubsection{Construction of Sobolev solution on an analytic domain}\label{sec_constr_soban}

Here we prove the local well-posedness Lemma~\ref{L01} in the case when $\Omega$ is analytic (in the sense that \eqref{def_an_domain}). To this end, we first show how to approximate Sobolev vector fields by analytic ones.

\begin{lemma}[Analytic approximation lemma]\label{L_an_approx}
Let $\Omega$ be a bounded and analytic domain (in the sense of \eqref{def_an_domain}). There exists $\varepsilon_0>0$ such that, for every $\varepsilon>0$, there is $S_\varepsilon \in B(\mathcal{H}_r,\mathcal{H}_r)$ with $\| S_\varepsilon \| \lec 1$ such that, for sufficiently small $\varepsilon >0$, $S_{\varepsilon } u \in \oY (\varepsilon_0 )$ and
\[
\| S_\varepsilon u - u \|_{H^r} \to 0
\]
as $\varepsilon \to 0$. Moreover, there exist an operator $T_\varepsilon$ with the same properties as $S_\varepsilon$, except that $T\in B(V,V)$, where  
\[
V\coloneqq \{ u\in H^r (\Omega ) \colon u\cdot \nn =0 \text{ on }\p \Omega \}.
\]
\end{lemma}
\begin{proof}
We denote by $E\in B(H^r (\Omega ), H^r (\R^3))$ the Sobolev extension operator (see \cite[Theorem~5.22]{AF}). Given $u\in \mathcal{H}_r$, we set
\[
u_\varepsilon \coloneqq \Phi (\varepsilon ) \ast Eu,
\]
where $\Phi(x, t)\coloneqq (4\pi t)^{-3/2} \mathrm{e}^{-|x|^2/4t} $. We let $S_\varepsilon u\coloneqq v$ to be the unique $H^r$ solution to the Stokes system
\eqnb\label{stokes}
\begin{split} -\Delta v + \nabla p &= -\Delta u_\varepsilon ,\\
\div\, v &=0\qquad \text{ in }\Omega,\\
v\cdot \nn &=0 ,\\
v\times \nn &= u_\varepsilon \times \nn \qquad \text{ on }\p \Omega
;
\end{split}
\eqne
see~\cite[Proposition~2.2 in Ch.~1]{T}. Note that $\| v \|_{H^r (\Omega )} \lec \| u_\varepsilon \|_{H^r (\Omega )} \leq \| u_\varepsilon \|_{H^r (\R^3 )} \leq \| E u \|_{H^r (\R^3 )} \lec \| u \|_{H^r (\Omega )}$, which shows that $S_\varepsilon u \in \mathcal{H}_r$ for every $\varepsilon >0$. Moreover,
\[
\| S_\varepsilon u - u \|_{H^r} \leq \| S_\varepsilon u - u_\varepsilon  \|_{H^r} +\|  u_\varepsilon - u \|_{H^r(\Omega )} \to 0
,
\]
as $\varepsilon \to 0$, where used the approximation property of the heat kernel (see Appendix~6.5.1 in \cite{OP}, for example) to obtain convergence of the second term. As for the second term we see that $w\coloneqq S_\varepsilon u - u_\varepsilon $ is a solution to
\eqnb\label{stokes1}
\begin{split} -\Delta w + \nabla p &= 0 ,\\
\div\, w &=-\div \,u_{\varepsilon }   \qquad \text{ in }\Omega,\\
w\cdot \nn &=-u_\varepsilon \cdot \nn ,\\
w\times \nn &= 0 \qquad \text{ on }\p \Omega,
\end{split}
\eqne
which, by \cite[Proposition~2.2 in Ch.~1]{T}, gives 
\[\| w \|_{H^r (\Omega )} \lec \|\div\, u_{\varepsilon}  \|_{H^{r-1} (\Omega )}+\| u_{\varepsilon} \cdot \nn   \|_{H^{r-1/2} (\p \Omega )}  \to \|\div\, u  \|_{H^{r} (\Omega )} +\| u \cdot \nn   \|_{H^{r-1/2} (\p \Omega )}=0\]
as $\varepsilon \to 0$, since $ \,u_\varepsilon \to  u $ in $H^{r} (\R^3)$.

Furthermore, $S_\varepsilon u \in \oY (\varepsilon_0 )$ since $u_\varepsilon \in Y(\R^3;\varepsilon_0 )$ for sufficiently small $\varepsilon_0 = \varepsilon_0 (\Omega )$ (see~\cite{KP2}) and then $S_\varepsilon u \in \oY (\Omega ;\varepsilon_0 )$ by analyticity of $\Omega $ and analyticity of solutions to the Stokes problem (see \cite{JKL1,JKL2} and \cite{KP2} for details).

As for operator $T_\varepsilon$, we apply the same construction, except that the second equations in \eqref{stokes} and \eqref{stokes1} are replaced by $\div\, v =\div u_\varepsilon - [\div \,u_{\varepsilon } ]_\Omega$ and $\div\, w =-[\div \,u_{\varepsilon } ]_\Omega$, respectively, where $ [\div \,u_{\varepsilon } ]_\Omega \coloneqq \frac{1}{|\Omega |} \int_{\Omega } \div \,u_\varepsilon $.
\end{proof}

Using Lemma~\ref{L_an_approx}, we can use analytic approximations of $U,u_0,b_0$. Namely, given $U\in C([0,T_0]; V)$ and $u_0,b_0\in \mathcal{H}_r$ we consider a sequence $\varepsilon_k \to 0$ and a sequence of approximations
\[
\us_{k } \coloneqq  T_{\varepsilon_k} \us,\qquad u_k \coloneqq S_{\varepsilon_k } u_0 , \,b_k \coloneqq S_{\varepsilon_k } b_0.
\]
For each $k$, we solve the MHD system \eqref{EQ07a} on $[0,T_0]$, perturbed around $\us_k$, with initial data $(u_k,b_k)$ to obtain a solution $\psi_k \in C([0,T_0]; H^r (\Omega ))$, and we consider the limit $k\to \infty$. We note that, for each $k$ the choice of $\tau$ in the analytic spaces \eqref{EQ19c} is defined by \eqref{tau_choice_actual}, except with $\tau_0$ replaced by~$\varepsilon_0$. In particular, we see that $H^r$  estimate \eqref{002b} remains uniform with respect to $k $ and so we can find a subsequence of $\psi_k$ which converges weakly-$*$ in $L^\infty ((0,T_0);H^r)$. We also note that $\p_t \psi_k$ is bounded in $C([0,T_0];H^{r-1})$, uniformly with respect to $k$, which lets us use the Aubin-Lions lemma (see~\cite[Theorem~2.1 in Section~3.2]{T}, for example) to obtain a solution $\psi \in C ([0,T_0];H^r)$ of the MHD system \eqref{EQ07a_repeat} on $[0,T_0]$, perturbed around $\us$ with initial data~$\psi_0$. The uniqueness of such solution can be obtain by a simple energy argument in $C([0,T_0];L^2)$. 

\subsubsection{Construction of Sobolev solution on Sobolev domain}\label{sec_constr_sobsob}

Here we prove the local well-posedness Lemma~\ref{L01} in full generality. Namely, we consider $r\geq 3$, a bounded $H^{r+2}$ domain $\Omega $, a background vector field $\us \in C([0,T_0]; H^{r+1} (\Omega ) )$, $T_0$ given by \eqref{gprime}, and, for every $u_0,b_0\in \mathcal{H}_r$, we construct a unique solution $\psi \in C([0,T_0];H^r (\Omega ))$ to~\eqref{EQ07a_repeat}. 

\noindent\texttt{Step~1.} For a given $\varepsilon>0$, we find an analytic domain $Q$ that approximates~$\Omega$.\\

Namely, we fix $\phi \in H^{r+2}$ such that $\Omega =\{ \phi >0 \}$ (recall~\eqref{sobolev_domain}). Given $\varepsilon>0$, we set \[\phi_\varepsilon \coloneqq \Phi (\varepsilon )\ast \phi,
\]
where $\Phi $ denotes the 3D heat kernel, and we set 
\[
Q \coloneqq \{ \phi_\varepsilon >0 \}.
\]
Then, $\phi_\varepsilon $ is analytic by properties of the heat kernel (see \cite{KP2}, for example), and so $Q$ is an analytic domain; recall \eqref{def_an_domain}. Moreover, $Q$ can be described locally as a graph of an analytic function, which can be proven using an analytic version of the Implicit Function Theorem; see~\cite{KP1}, and $Q \to \Omega  $ in $H^{r+2}$ as $\varepsilon \to 0$, by following the analysis of \cite[Section~6.3]{A}. We emphasize that a given Sobolev domain is denoted by $\Omega$, while the approximate analytic domain, for a given $\varepsilon$, is denoted by $Q$. \\

\noindent\texttt{Step~2.} We construct a $H^{r+2}$ diffeomorphism $\eta \colon \Omega \to Q$.\\

To this end, we first let $\beta >0$ be a small number, to be determined later. Let $\xi \colon \R^3 \to [0,1]$ be a smooth function such that $\xi=1$ on $B(0,1-\beta )$ and $\xi =0$ on $B(0,1 )^c$.

Let $B_1, \ldots , B_L$ be a collection of balls that cover $\p \Omega $ and are such that, for each fixed $l\in \{ 1,\ldots , L\}$, the surface $\p \Omega\cap B_l$ is (up to rotation) a graph of a $H^{r+2}$ function $F_0$ that is bounded away from~$0$. Let $\varepsilon>0$ be sufficiently small so that the collection also covers $\p Q $ and, similarly, for each $l$, $\p \Omega \cap B_l$ is (after the same rotation) a graph of a $H^{r+2}$ function $F$ . We now fix $\beta>0$ sufficiently small so that the same is true for the collection $(1-\beta)B_1,\ldots , (1-\beta )B_L$.
Set $\xi_l (x) \coloneqq \xi (x-x_l)$, where $x_l$ is the center of~$B_l$. 
We define $\eta$ inductively: $\eta_0 \coloneqq \id$, and, for $l=\{ 1,\ldots , L \}$, we redefine $F_0$ and $F$ to be the graphs of $\p \eta_{l-1} (\Omega)\cap B_l$ and $\p \eta_{l-1} (\Omega)\cap B_l$, respectively, and we let 
\[ G_l \colon B_l \to  B_l \]
be such that $G_l$ is a diffeomorphism of $\p \eta_{l-1} (\Omega)\cap B_l$ onto $\p \eta_{l-1} (Q )\cap B_l$. Moreover, since both $\Omega$, $Q$ are $H^{r+2}$ domains, we can guarantee that $G_l \in H^{r+2}(B_l, B_l )$. Such $G_l$ can be constructed, for example, by considering the corresponding rotation and translation and then defining $G_l$ via the quotient~$F/F_0$. Since $F,F_0\in H^{r+2}$, and since $F_0$ is bounded away from $0$, the same is true of the quotient, and of the appropriate rotation and translation.

We now set
\[
\eta_{l} \coloneqq \left( \xi_l G_l + (1-\xi_l )\id \right) \circ \eta_{l-1}
\]
and $\eta \coloneqq \eta_{L}$. An inductive argument shows that $\eta \in H^{r+2}$ and that
\eqnb
\| \eta - \id \|_{H^{r+2}} \to 0\qquad \text{ as }\varepsilon \to 0.\vspace{0.6cm}
\eqne

\noindent\texttt{Step~3.} Given initial data $(v_0,b_0)$ on $\Omega$ and background flow $\us\in C([0,T_0];H^{r+1} (\Omega))$, we construct approximate data and background flow on~$Q$.\\

Namely, we set 
\[
v_0 \coloneqq u_0 \circ \eta^{-1}
\andand
g_0\coloneqq b_0 \circ \eta^{-1}
\]
and we use Lemma~\ref{L_CS} to construct solutions $U_0,B_0 \in H^r$ to the problems
\[
\begin{cases}
\div\, U_0 =0,&\\
\curl\, U_0 =\curl_a\, v_0 \qquad &\text{ in }\Omega,\\
U_0\cdot \nn =0\qquad &\text{ on }\p \Omega,
\end{cases} \qquad \begin{cases}
\div\, B_0 =0,&\\
\curl\, B_0 =\curl_a\, g_0 \qquad &\text{ in }\Omega ,\\
B_0\cdot \nn =0\qquad &\text{ on }\p \Omega ,
\end{cases} 
\]
respectively, satisfying 
\[
\| U_0 \|_{H^r} \lec_{\Omega } \|  v_0 \|_{H^r } \lec \|  u_0 \|_{H^r (\Omega )} \andand \| B_0 \|_{H^r} \lec_{\Omega } \|  g_0 \|_{H^r } \lec \|  b_0 \|_{H^r (\Omega )} ,
\]
where
\[
a\coloneqq (\na \eta )^{-1}
\]
and 
\[
\curl_a w \coloneqq \epsilon_{ijk} a_{lj}\p_l w_k ,
\]
using the summation convention. Similarly, we set 
\[
\Delta_a \coloneqq \p_j (a_{jl } a_{kl} \p_k (\cdot ))
.
\]
We now define the background flow $\us$ on the approximate domain~$Q$. To this end we recall Section~\ref{sec_uS} that $U$ is of the form
\eqnb\label{us_form}
U(x,t) = \sum_{i=1}^l h_i (t) \theta_i (x)
,
\eqne
where each $\theta_i\in H^{r+2} (\Omega )$ is a solution of the problem \eqref{theta_problem} on $\Omega$, i.e.,
\eqnb\label{theta_problem_restate_onOmega0}
\begin{cases}
\Delta \theta =0 \qquad &\text{ in }\Omega',\\
\p_{\nn} \theta =0 \qquad &\text{ on }\p \Omega' \setminus \Gamma,\\
\nabla \theta (x) =v,
\end{cases}
\eqne
 for some $x\in \p \Omega'$ and some $v\in T_{x }(\p \Omega')$. In order to obtain an approximation of $U$, which is defined on $Q \times [0,T_0]$, and is of the same form as \eqref{us_form},  we will need the following statement.
 
\begin{lemma}\label{L66}
Let $\overline{x}\in \p \Omega$. There exists $\varepsilon_0$ such that for each $\varepsilon\in (0,\varepsilon_0]$ the following holds: For every $v\in T_{\overline{x}}(\p \Omega_0)$ there exists a solution $\theta_a$ to the problem
\[
\begin{cases}
\Delta_a \theta_a =0\qquad &\text{ in }\Omega',\\
\p_{\nn} \theta_a =0 \qquad &\text{ on }\p \Omega'\setminus \Gamma ,\\
\nabla \theta_a (\overline{x})=v .
\end{cases}
\] 
Moreover, $\| \theta_a - \theta \|_{H^{r+2}}\to 0$ as $\varepsilon \to 0$.
\end{lemma}
\begin{proof} We first suppose that $\varepsilon$ is sufficiently small so that
\eqnb\label{psia9}
\| I-a \|_{H^{r+2}} \leq 1
.
\eqne
Given $w\in T_{\overline{x}}(\p \Omega)$, let $\theta$ be a solution of \eqref{theta_problem_restate_onOmega0} with $\nabla \theta (\overline{x}) = w$. We will find a solution $\psi\in H^r (\Omega)$ to the problem
\eqnb\label{psia}
\begin{cases}
\Delta_a \psi =-(\Delta-\Delta_a) \theta \qquad &\text{ in }\Omega,\\
{\nn}\cdot (a^T  \nabla \psi) =0 \qquad &\text{ on }\p \Omega \setminus \Gamma,
\end{cases}
\eqne
such that
\eqnb\label{psia_conv}
\| \psi \|_{H^{r+2}} \lec \| I-a \|_{H^{r+2}} \| \theta \|_{H^{r+2}}.
\eqne
Then taking $\theta_a \coloneqq \theta - \psi$ gives the claimed solution.

In order to find $\psi$, we first let $\chi \in C_0^\infty (\Gamma)$ be such that $\int_\Gamma \chi =1$, and we consider the Neumann problem
\eqnb\label{psia1}
\begin{cases}
\Delta \psi = (\Delta - \Delta_a ) (\psi - \theta ) \qquad &\text{ in }\Omega,\\
\p_{\nn} \psi = {\nn}\cdot ((I-a^T)  \nabla \psi)+ c \chi  \qquad &\text{ on }\p \Omega,
\end{cases}
\eqne
where $c\in \R$.
Note that \eqref{psia1} is uniquely solvable if and only if 
\eqnb\label{psia2}
c = -\int_{\p \Omega } \nn\cdot ((I-a^T)  \nabla \psi) + \int_{\Omega} (\Delta - \Delta_a ) (\psi - \theta ).
\eqne
We are thus looking for a solution $(\psi,c)\in H^{r+2} \times \R$ to \eqref{psia1}--\eqref{psia2}. We set $(\psi_0,c_0)\coloneqq (0,0)$, and, for $k\geq 1$ we let 
\[
c_k\coloneqq -\int_{\p \Omega } \nn \cdot ((I-a^T)  \nabla \psi_{k-1}) + \int_{\Omega} (\Delta - \Delta_a ) (\psi_{k-1} - \theta ),
\]
and we let $\psi_k \in H^{r+2} $ be the solution to
\eqnb\label{psia3}
\begin{cases}
\Delta \psi_k = (\Delta - \Delta_a ) (\psi_{k-1} - \theta ) \qquad &\text{ in }\Omega,\\
\p_{\nn} \psi_k = {\nn}\cdot ((I-a^T)  \nabla \psi_{k-1})+ c_k \chi  \qquad &\text{ on }\p \Omega.
\end{cases}
\eqne
We have that 
\eqnb\label{psia4}
|c_k|+\| \psi_k \|_{H^{r+2}} \lec \| I-a \|_{H^{r+2}} \left( \| \psi_{k-1} \|_{H^{r+2}}+\| \theta \|_{H^{r+2}}\right) ,
\eqne
due to~\eqref{psia9}. Since $\psi_{k+1} - \psi_k$ is a solution to
\eqnb\label{psia5}
\begin{cases}
\Delta (\psi_{k+1} - \psi_k ) = (\Delta - \Delta_a ) (\psi_{k} - \psi_{k-1}  ) \qquad &\text{ in }\Omega,\\
\p_{\nn} (\psi_{k+1} - \psi_k ) = {\nn}\cdot ((I-a^T)  \nabla (\psi_{k} - \psi_{k-1} ) ) + (c_k-c_{k-1})  \chi  \qquad &\text{ on }\p \Omega,
\end{cases}
\eqne
we have that
\[
|c_{k+1} - c_k | + \| \psi_{k+1} - \psi_k  \|_{H^{r+2}} \lec \| I-a \|_{H^{r+2}} \left( |c_{k} - c_{k-1} | + \| \psi_{k} - \psi_{k-1}  \|_{H^{r+2}} \right) 
,
\]
for all $k\geq 1$. Thus, if $\varepsilon_0>0$ is sufficiently small, then $(\psi_k,c_k) $ is Cauchy in $H^{r+2} \times \R$, and so $\| \psi_k - \psi \|_{H^{r+2}} , |c_k -c | \to 0$ as $k\to \infty$, for some $\psi \in H^{r+2}, c\in \R$. We can thus take $k\to \infty $ in \eqref{psia3} to see that $(\psi,c)$ solve \eqref{psia1}--\eqref{psia2}, which implies \eqref{psia}, as $\chi=0 $ on $\p \Omega_0 \setminus \Gamma$. Letting $k\to \infty $ in \eqref{psia4}, we  obtain \eqref{psia_conv}, as required.
\end{proof}
The point of the lemma is, of course, that, if $\theta_a$ is given by Lemma~\ref{L66} then $\theta_a \circ \eta^{-1}$ satisfies \eqref{theta_problem_restate_onOmega0} on $Q$, with a slightly different $x_i$ and  $v_i$, if $\varepsilon>0$ is sufficiently small.

To be precise, we consider the finite sequence $\{ \eta (x_m )\}_{m=1}^M$ (recall~\eqref{def_partition}), and we let $\theta^{(a)}_i$ be given by Lemma~\ref{L66} with $x_i\in \p \Omega'$, $v_i \in T_{x_i}(\p \Omega)$. Then we set
\[
V(x,t) \coloneqq \sum_{i=1}^l h_i (t) \theta^{(a)}_i\circ \eta^{-1} (x),
\]
where $h_i$ are the same as in~\eqref{us_form}. We note that this way $V(x,t)$ solves the 3D incompressible forced Euler equations with forcing supported in $\eta (\Omega \setminus \overline{\Omega'})$, with non-penetration boundary condition on $\p \Omega'$. 
Moreover the path $\eta (F([0,1]))$ might not be the same as $\phi^V (\eta (a))$. This is not a problem, as the path is not relevant for the local well-posedness result (Lemma~\ref{L01}).
However, we have
\eqnb\label{V_goesto_U}
\| U- V \|_{C([0,T_0];H^{r+2})}\to 0\qquad \text{ as }\varepsilon \to 0,
\eqne
due to the last claim of Lemma~\ref{L66}.\\

\noindent\texttt{Step~4.} We take the limit $\varepsilon \to 0$ to conclude the proof of the local well-posedness Lemma~\ref{L01}.\\

Namely, we consider the background flow $V$ defined in the previous step to apply the construction of Section~\ref{sec_constr_soban} to obtain a $(v,g)\in C([0,T_0];H^r)$ to the  MHD system \eqref{EQ07a_repeat}, perturbed around $V$ on $Q$ with initial data $(U_0,B_0)$. Since taking the limit $\varepsilon\to 0$ we have $U_0 \to u_0$, $B_0 \to b_0$ as $\varepsilon \to 0$ and recalling \eqref{V_goesto_U}, we can use the uniform $C([0,T_0];H^{r})$ estimate \eqref{002b} as well as an Aubin-Lions argument (as in Section~\ref{sec_constr_soban} above) to obtain a limit $\psi \in C([0,T_0];H^r (\Omega ))$ satisfying \eqref{EQ07a_repeat} around background velocity field $\us$ and with initial condition $(u_0,b_0)$. Uniqueness follows by a similar energy estimate in $C([0,T_0];L^2 )$ as in Section~\ref{sec_constr_soban}.  

\subsection{The cancelling operator}\label{sec_cancel}

Here we construct the cancelling operator $T$ of Step~6 (recall~\eqref{opT_prop}). Namely, we suppose that $\Omega $ is a $H^{r+2}$ domain, that $\widetilde{W}, W\subset \Omega$ are  $H^{r+2}$ subdomains such that $\widetilde{W}\subset \subset W$, and that $\chi \in C^\infty (\Omega ; [0,1])$. Then we set 
\eqnb\label{what_is_V}
V \coloneqq \{ w\in \mathcal{H}_r \colon \text{ either } \chi (x) =1 \text{ or } w(x)=0 \quad \text{ for every } x\in W\setminus \widetilde{W}  \}
\eqne
(recall \eqref{def_Hr}, where $\mathcal{H}_r \coloneqq \{ u\in H^r (\Omega ) \colon \div u =0\text{ in }\Omega , \, u\cdot n =0 \text{ on }\p \Omega  \}$), 
and we prove the following.
\begin{lemma}
There exists $T\in B(V,\mathcal{H}_r)$ such that $Tb =b$ on $\Omega \setminus W$.
\end{lemma}
\begin{proof}
Let $\widetilde{W}\subset \subset W$ be a $H^{r+2}$ subdomain such that \eqref{what_is_V} is valid with $W$ replaced by~$\overline{W}$. Let $\chi_0 ,\chi_1, \ldots , \chi_r  \in C^\infty (\Omega ; [0,1])$ be sequence of cutoff functions such that
\[
 \Omega\setminus W \subset \subset \{ \chi_r =1 \} \subset \subset \supp\, \chi_r   \subset \subset \ldots \{ \chi_0 =1 \} \subset \subset \supp\, \chi_0   \subset \subset \Omega \setminus \widetilde{W}.
\]

We note that since
\[
\int_{\widetilde{W}}  \nabla \chi \cdot b = \int_{\widetilde{W}} \div \, (b \chi ) = \int_{\p {\widetilde{W}}} b\chi \cdot n =0,
\]
by assumption and by $b\cdot n=0$ on $\p \Omega$, Lemma~\ref{L_CS} shows that there exists  $\psi$ such that 
\eqnb\label{make_div0}
\begin{cases}
\curl\, \psi  =0 \qquad &\text{ in }{\widetilde{W}} ,\\
\div \, \psi =\nabla \chi \cdot b   &\text{ in } {\widetilde{W}} ,\\
\psi \cdot \nn =0 &\text{ on }\p {\widetilde{W}},
\end{cases}
\eqne
and that $\| \psi \|_{H^{r}({\widetilde{W}})} \lec \| b \|_{H^r({\widetilde{W}})}$. Note that 
\[
\psi_0\coloneqq  b- \psi \in L^2 (\Omega )
\]
is divergence-free, $\psi_0 = b $ on $\Omega \setminus \widetilde{W} \supset \supp \, \chi_0$, and 
\[
\| \psi_0 \|_{L^2 (\Omega )} \lec \| b \|_{H^r (\Omega )},
\]
but it is not clear that $\psi_{0} \in \mathcal{H}_r$ (i.e., across $\p \widetilde{W}\cap \Omega $). To this end, we describe an iteration procedure which bootstraps the regularity up to $H^r (\Omega )$. Namely, for each $k=1,\ldots, r$, given $\psi_{k-1} \in \mathcal{H}_{k-1}$ such that
\[
\psi_{k-1}  =  b\qquad \text{ on } \{  \chi_{k-1} =1 \} ,
\]
 we use Lemma~\ref{L_CS} (with $l\coloneqq k$) to obtain a unique solution $\phi_k\in H^k$ be the solution to the problem
\eqnb\label{increase_Hk}
\begin{cases}
\curl\, \phi_k  =\psi_{k-1} \qquad &\text{ in }{\Omega } ,\\
\div \, \phi_k =0    &\text{ in } {\Omega} ,\\
\phi_k \times \nn =0 &\text{ on }\p {\Omega},
\end{cases}
\eqne
with $\| \phi_k \|_{H^{k}} \lec \| \psi_{k-1} \|_{H^{k-1}}$. 
We set 
\[
\psi_k \coloneqq \curl\, (\phi_k \chi_k) = \chi_k  \underbrace{\curl\, \phi_k }_{=\psi_{k-1}=b} + \phi_k \times \nabla \chi_k.
\]
Clearly, $\psi_k = b$ on $\{ \chi_k=1\} $ and 
\[
\| \psi_k \|_{H^k (\Omega )} \lec_k \| b \|_{H^r (\Omega )} + \| \phi_k \|_{H^k (\Omega )}. 
\]
Moreover $\div \,\psi_k=0$ and, on $\p \Omega $,
\[
 \psi_{k} \cdot \nn = \curl\, (\phi_k \chi_k )\cdot \nn = \chi_k \underbrace{(\curl \phi_k )}_{=b}\cdot \nn + (\phi_k \times \nabla \chi_k ) \cdot \nn=\chi_k b\cdot \nn  - (\phi_k \times\nn )\cdot \nabla \chi_k =0.
  \]
  Thus $\psi_k \in \mathcal{H}_k$ for each $k=1,\ldots , r$. The lemma follows by setting $Tb \coloneqq\psi_r$.
\end{proof}
The above proof implies also that given $b_0\in \mathcal{H}_r$ we can cut it off in $\Omega\setminus \overline{\Omega'}$ arbitrarily close to $\Omega'$.
\begin{corollary}\label{cor_wlog}
Given $\delta >0$ there exists $T\in B(\mathcal{H}_r, \mathcal{H}_r)$ such that $Tb=0$ on $\Omega \setminus  \overline{\Omega'+B(\delta )}$
\end{corollary}
\begin{proof}
We let $\chi \in C^\infty$ be such that $\chi=1 $ on $ \Omega'+B(\delta /2)$ and $\chi =0$ on $\Omega \setminus  \overline{\Omega'+B(\delta )}$. We then follow the same procedure as in the proof above, by first making $b\chi$ diveregence-free as in \eqref{make_div0}, and then iterating regularity by considering a sequence of cutoff functions and increasing the regularity up to $H^r$ as in \eqref{increase_Hk}.
\end{proof}

\section{2D controllability}\label{sec_2d}

Here we prove Theorem~\ref{T01} in the case $n=2$. As in Steps 1 and 2 in Section~\ref{sec_3d}, we note that we can assume that $b_1=0$, $u_1$ is arbitrary, and that $T_0$ is arbitrarily large.

We now note that since we can assume that $u_0$ and $b_0$ have compact support in $\Omega $ (recall Corollary~\ref{cor_wlog} above), we can replace the ambient domain $R$ by a $H^{r+2}$ domain that is periodic in $x_1$ with some period~$L$. Denote by 
\[
\Omega \coloneqq \{ (x,y) \in \R^2 \colon (x,y) = (x_0+kL,y_0)\in R , \,k\in \ZZ \}
\]
the periodic extension of~$R$.

We show in Section~\ref{sec_conform} below that $\Omega$ can be mapped via a conformal and periodic map onto a unit periodic channel $S \coloneqq \{ z\in \C \colon |\im \,z| <1 \}$. Namely, we say that $F\colon R\to S $ is periodic if there exists $M>0$ such that
\[
F(z+L)=F(z)+M
,
\]
for every $z=x+iy \in \C$ such that  $(x,y)\in R$, and we show the following.
\begin{proposition}\label{prop_conformal}
There exists a periodic conformal mapping $F\colon R\to S$.
\end{proposition}

Given Proposition~\ref{prop_conformal}, we let $U\colon R\to \R^2$ denote the vector field obtained by mapping conformally  the standard uniform shear flow, i.e.,
\[
U(x)\coloneqq \nabla \left(\mathrm{Re}\, F(z) \right).
\]
Note that there exist $T>0$ such that each particle in $R$ is transported via $U$ away from $R$ after time $T>0$, i.e.,
\eqnb\label{time_to_go}
\Phi^{U}(x,T) \not \in R\qquad \text{ for }x\in R, t>T.
\eqne
In order to see this, assume that $(x_0,y_0)$ belongs to the left part of the boundary $\p R$ and consider a particle trajectory $(x(t),y(t))$ of velocity field $U$, i.e.
\[
\begin{split}
\begin{pmatrix}
\dot x (t) \\ \dot y(t) 
\end{pmatrix} &= U (x(t),y(t))= \begin{pmatrix}
\p_x u(x(t),y(t))\\
\p_y u (x (t),y(t) ),
\end{pmatrix} \\
(x(0),y(0)) &= (x_0,y_0),
\end{split}
\]
where we used the notation $F(z)= u(x,y) + i v(x,y)$. Then letting $p(t)+iq(t) \coloneqq F(x(t)+iy(t))= u(x(t),y(t))+iv(x(t),y(t))$, we see that
\[
\dot p (t) = \p_x u\cdot \dot x +\p_y u \cdot \dot y = ( \p_x u )^2 + (\p_y u )^2 = |F'(z(t))|^2
.
\] 
Thus, since $F\colon R \to S$ is a conformal mapping (which can be extended to $\p R$), we have $F (z) \ne 0$ for all $z\in \overline{R}$, and so, by compactness, there exists $\eta >0$ such that $|F'|^2 \geq \eta $ in~$\overline{R}$. This shows that $(p(t),q(t))\not \in S $ for $t\geq C \eta^{-1}$, and consequently \eqref{time_to_go} holds with $T\coloneqq C \eta^{-1}$.

We can thus denote by $\psi$ the solution of the perturbed MHD system \eqref{EQ07a} on $\Omega$, around $U$ with initial condition $\psi(0)=\psi_0\coloneqq (u_0,b_0)\chi_R$,  so that $u_0=b_0=0$ in the periodic extension of~$R$. We now take $\| \psi_0 \|_{H^r}$ sufficiently small so that particle trajectories of $u(\cdot, t) + U$ remain close to the particle trajectories of $U$ for $t\in [0,T]$, so that 
\[
\Phi^{u(\cdot, t) + U}(x,T) \not \in R\qquad \text{ for }x\in R, t>T.
\]
This guarantees that $b$ has been transported away from $R$, and so, restricting $\psi (\cdot , t)$ to $\Omega'$ for $t\in [0,T]$ completes the proof of Theorem~\ref{T01} in the case $n=2$.

\subsection{Periodic conformal mapping}\label{sec_conform}
 
Here we prove Proposition~\ref{prop_conformal}. Without loss of generality, we may assume that $0\in R$, and $[-iD, iD ] \subset R$ with $-iD,iD \in \p R$.\\

We first note that, given $z_0 \in D$, the linear fractional transformation
\eqnb\label{conformal_disc}
G_{z_0}(z) \coloneqq \frac{z-z_0}{1-\overline{z_0}z}
\eqne
maps conformally $D$ onto $D$ in a way that $G_{z_0}(z_0)=0$, $G_{z_0}'(z_0)>0$, and moreover circles that are symmetric with respect to $[0,z_0]$ and pass through the inversion of $z_0$, $\mathrm{inv}\,z_0 \coloneqq \overline{z_0}^{-1}$, become parallel lines; see Figure~\ref{fig_lft}.

\includegraphics[width=0.8\textwidth]{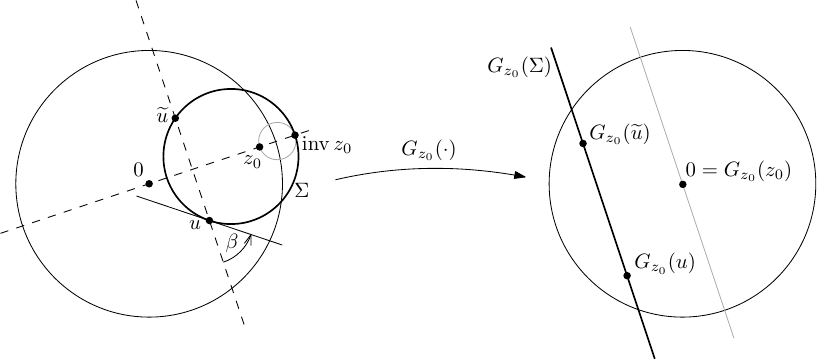}
  \captionof{figure}{A sketch of the properties of the mapping $G_{z_0}$. Note that the mapping rotates $u$ by angle $-\beta= \mathrm{Arg}\, G_{z_0}(u)$.}\label{fig_lft} 

We also recall the following.

\begin{theorem}[The Riemann Mapping Theorem]\label{thm_riemann}
Given any simply-connected $Q\subset \C$, such that $Q\ne \C$, and $z_0\in Q$, there exists a unique biholomorphic bijection $f\colon Q\to D$ such that $f(z_0)= 0$ and $f'(z_0) > 0$.
\end{theorem}
Remark: The claim of the Riemann Mapping Theorem remains true if the condition $f'(z_0)>0$ is replaced by ${\rm e}^{-i\alpha } f'(z_0) >0$ for any fixed $\alpha \in [-\pi, \pi )$ (namely by $\mathrm{Arg}\, f'(z_0) = \alpha$).

We will call such $f$ the Riemann mapping of $Q$ with respect to $z_0$ and angle~$\alpha$.

We note that the Riemann mapping $g$ of $S$ with respect to $0$ and angle $0$ is symmetric with respect to both the real and imaginary axis, namely 
\eqnb\label{sym_riem}
\overline{g(z)} = g(\overline{z}), \qquad \text{ and }\qquad -\overline{g(z)} = g(-\overline{z})\quad \text{ for all }z\in S,
\eqne
which follows from the uniqueness part in Theorem~\ref{thm_riemann}.

As a consequence, we have the following.

\begin{corollary}\label{cor_a_mapping}
Given $z_0\in R$ and $\alpha \in \R$ there exists a unique biholomorphic bijection $F\colon R \to S$ such that $F(z_0)=0$ and ${\rm e}^{-i\alpha }F'(z_0) >0$.
\end{corollary}

\begin{proof}
We simply use the Riemann Mapping Theorem twice, first with $R$ and $z_0$ to obtain $f\colon R \to D$ and then with $S$ and $0\in S$ to obtain $g\colon S\to D$. The required mapping is then $F\coloneqq g^{-1} \circ R_{-\alpha} \circ f \colon R\to S$.
\end{proof}

In order to prove Proposition~\ref{prop_conformal}, we now show that one can find $z_0\in R$, $\alpha \in \R$ such that the mapping in Corollary~\ref{cor_a_mapping} is periodic.

\begin{proof}[Proof of Proposition~\ref{prop_conformal}.]
We will show, in Step~1 below, that there exists $\delta\in (-D,D)$ and $\alpha \in \R$ such that the Riemann mapping $w$ with respect to $z_0 \coloneqq i\delta $ and the angle $\alpha$ satisfies $w(L+i\delta )>0$ and ${\rm e}^{-i\alpha }w'(L+i\delta )>0$. Given such $w$, we have that
\eqnb\label{conformal_what_is_F}
F= g^{-1}\circ w
\eqne
is the mapping from Corollary~\ref{cor_a_mapping}; see \eqref{sym_riem} as well as Figure~\ref{fig_periodic1} below. We will show, in Step~2 below, that such $F$ is periodic, proving Proposition~\ref{prop_conformal}.\\

\noindent\texttt{Step~1.} We construct~$w$.\\

We first denote by $f$ the Riemann mapping of $R$ with respect to $0$ and angle $0$, which maps the segment $R_0$ of $R$ between $0$ and $L$ onto a region in $D$; see Figure~\ref{fig_periodic}.

\includegraphics[width=\textwidth]{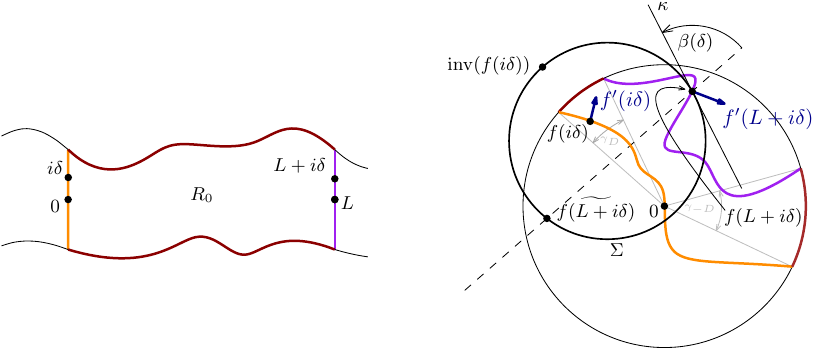}
  \captionof{figure}{A sketch of the Riemann mapping $f\colon R \to D$ with respect to $0$ and angle~$0$.}\label{fig_periodic}

We denote by $\widetilde{f(L+i\delta )}\in \C$ the symmetric image of $f(L+i\delta )$ with respect to the line $[0,f(i\delta )]$, and we denote by
\[ \Sigma \subset \C \quad \text{ the circle passing through points }\quad f(L+i\delta ), \widetilde{f(L+i\delta)}, \mathrm{inv}(f(i\delta ))
;
\]
see Figure~\ref{fig_periodic}. We also denote by $\kappa$ the tangent line to $\Sigma$ at $f(L+i\delta )$ and by $\beta = \beta (\delta ) \in (-\pi , \pi )$ the angle from the line $[f(L+i\delta ), \widetilde{f(L+i\delta )}]$ to the tangent line $\kappa$, as in Figure~\ref{fig_periodic}. 
We now consider $h\coloneqq [-D, D ] \to [-2\pi , 2\pi ]$, defined by
\[
h(\delta ) \coloneqq \beta (\delta )+ \mathrm{Arg}\,f'(L+i\delta ) - \mathrm{Arg}\, f'(i\delta )
\] 
Note that $\mathrm{Arg}\, f'(i\delta )$ is the angle that can be obtained by rotating the angle of the tangent to the curve $f(\{ i\eta \colon \eta \in [-D , D] \})$ at $f(i\delta )$ by $-\pi/2$. (Indeed, the rotation at each point is given by $\mathrm{Arg}\, f'$ and so the claim can be seen simply from the chain rule $\frac{\d }{\d \delta } f(i\delta ) = f' (i\delta ) i $, noting that $\frac{\d }{\d \delta } f(i\delta )$ is a tangent vector at $f(i\delta )$, $\delta \in [-D , D]$.)
Thus, since the theorem of Kellogg-Warschawski \cite[Theorem~3.6]{P} gives that $f$ remains conformal up to the boundary, we obtain that 
\[\mathrm{Arg}\,f'(L+iD ) - \mathrm{Arg}\, f'(iD ) = \gamma_D, 
\]
where $\gamma_D$  denotes the angle of the sector of $\p D$ which is the image of the part of $\p R$ joining $iD $ and $L+iD$, and similarly for $\delta = -D$;  see Figure~\ref{fig_periodic}.

Moreover, since for $\delta = D$ we have $\Sigma = \p D $ (and similarly at $\delta =-D$, only with the opposite orientation), the tangent-chord theorem gives that
\[
\beta (D) = \pi -\gamma_D,\qquad \beta (-D) = -\pi + \gamma_{-D}.
\] 
This shows that
\[
h(-D) = -\pi ,\qquad h(D) = \pi.
\]
Thus, since $h$ is continuous, the Darboux property implies the existence of $\delta \in (-D, D)$ such that $h(\delta )=0$. We now fix such $\delta$ and set
\[
z_0 \coloneqq i\delta .
\]

This shows that the mapping $ G_{f(z_0 )} \circ f \colon R\to D$ maps $z_0 $ onto $0$ and its derivative has the same argument at $z_0$ and at $L+z_0$. Indeed, the mapping $G_{f(z_0) } $ maps $\Sigma$ onto a line perpendicular to $[0,f(i\delta )]$, which changes the angle at $f(L+i\delta )$ by $-\beta (\delta )$, recall Figure~\ref{fig_lft} and see Figure~\ref{fig_periodic1}.

\includegraphics[width=\textwidth]{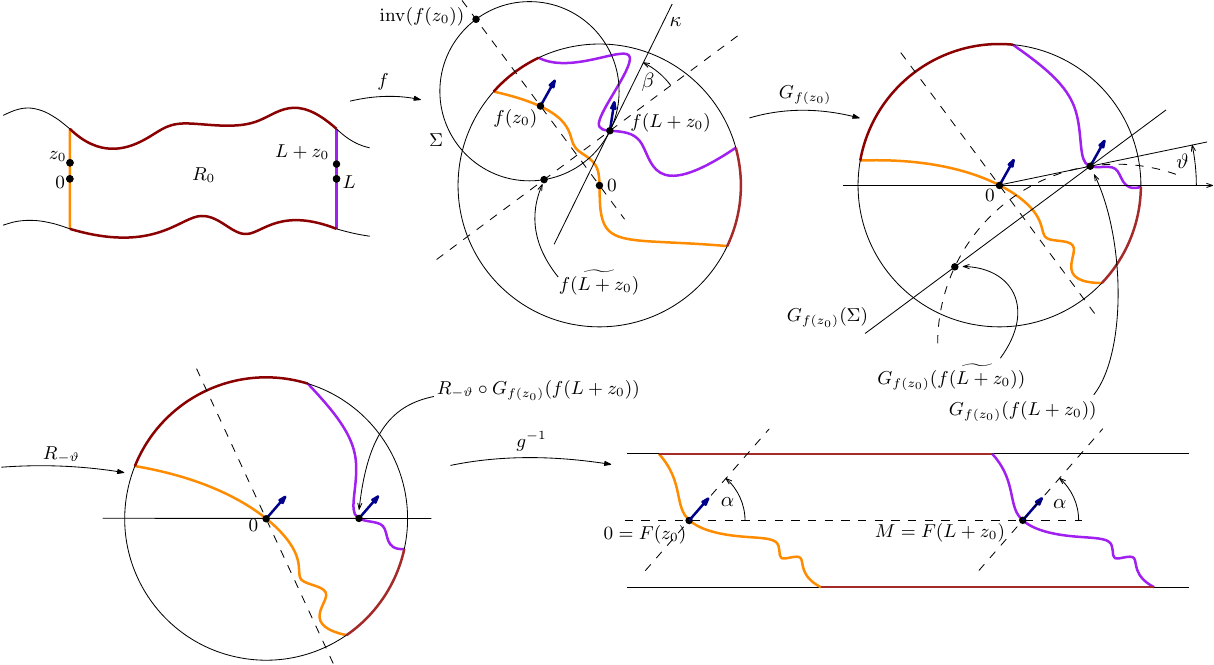}
  \captionof{figure}{A sketch of the Riemann mapping $w\colon R \to D$, and $F=g^{-1}\circ w \colon R\to D$.s}\label{fig_periodic1} 

We now let 
\[
\vartheta \coloneqq \mathrm{Arg}\, G_{f(z_0)} (f(L+z_0)  ),
\]
and $w$ can be obtained by
\[
w\coloneqq R_{-\vartheta } \circ G_{f(z_0 )} \circ f
;
\]
see Figure~\ref{fig_periodic1}.\\

\noindent\texttt{Step~2.} We show that $F$ is periodic.\\

We first note, from properties of $w$, that
\[
F(L+z_0 )>0,\qquad {\rm e}^{-i\alpha } F' (z_0+L ) >0
;
\]
see Figure~\ref{fig_periodic1}. Let 
\[
G(z) \coloneqq F(L+z)-F(L+z_0)
\]
Then $G\colon R\to S$ is biholomorphic, $G(z_0 ) =0$ and ${\rm e}^{-i\alpha } G'(z_0 ) ={\rm e}^{-i\alpha } F'(L+z_0 )$, and so the uniqueness part of Corollary~\ref{cor_a_mapping} implies that $F=G$. Equivalently $F$ is periodic with $M\coloneqq F(L+z_0)$, as required.
\end{proof}

\section*{Acknowledgments}
IK was supported in part by the
NSF grant
DMS-2205493.
This material is based upon work supported by the National Science Foundation under Grant No.~DMS-1928930 while the authors were in residence at the Simons Laufer Mathematical Sciences Institute (formerly MSRI) in Berkeley, California, during the summer of 2023. This work was partially supported by the Thematic Research Programme, University of Warsaw, Excellence Initiative Research University.



\colb

\colb

\ifnum\sketches=1
\newpage
\begin{center}
  \bf   Notes?\rm
\end{center}
\huge
\colb

\fi


\small
\begin{thebibliography}{[cmsz]}

\bibitem[A]{A} 
C.A.~Antonini,
  \emph{Smooth approximation of {L}ipschitz domains, weak
  curvatures and isocapacitary estimates},
  Calc.\ Var.\ Partial Differential Equations~\textbf{63} (2024), no.~4, Paper No. 91,~34.

\bibitem[AF]{AF} 
R.~A.~Adams and J.~J.~F.~Fournier,
  \emph{Sobolev Spaces},
 Second Edition, Pure and Applied Mathematics Series, Elsevier, 2003.

\bibitem[B1]{B1}
M.~E. Bogovski\u{\i}.
\newblock Solution of the first boundary value problem for an equation of
  continuity of an incompressible medium.
\newblock {\em Dokl. Akad. Nauk SSSR}, 248(5):1037--1040, 1979.

\bibitem[B2]{B2}
M.~E. Bogovski\u{\i}.
\newblock Solutions of some problems of vector analysis, associated with the
  operators {${\rm div}$} and {${\rm grad}$}.
\newblock In {\em Theory of cubature formulas and the application of functional
  analysis to problems of mathematical physics}, volume 1980 of {\em Trudy Sem.
  S. L. Soboleva, No. 1}, pages 5--40, 149. Akad. Nauk SSSR Sibirsk. Otdel.,
  Inst. Mat., Novosibirsk, 1980.

\bibitem[C1]{C1} 
J.-M.~Coron,
  \emph{On the controllability of {$2$}-{D} incompressible perfect fluids},
  J.~Math.\ Pures Appl.~(9)~\textbf{75} (1996), no.~2, 155--188.



\bibitem[C2]{C2} 
J.-M.~Coron,
  \emph{Control and nonlinearity},
  Mathematical Surveys and Monographs, vol.~136, American Mathematical Society, Providence, RI, 2007.

\bibitem[C3]{C3} 
J.-M.~Coron,
  \emph{On the controllability of the 2D incompressible Navier-Stokes equations with the Navier boundary conditions},
 ESAIM Control Optim.\ Calc.\ Var.~\textbf{1} (1996), 35--75.
  
\bibitem[C4]{C4} 
J.-M.~Coron,
  \emph{Contr\^olabilit\'e exacte fronti\`ere de l'\'equation d'Euler des fluides parfaits incompressibles bidimensionnels},
 C. R. Acad. Sci. Paris, t.~317, S\'erie~I, pp.~271--276, 1993. 
  
\bibitem[C5]{C5} 
J.-M.~Coron,
  \emph{Return method: application to controllability}, {\em
 S\'eminaire sur les \'Equations aux {D}\'eriv\'ees
              {P}artielles, 1992--1993}, Exp. No. XIV, 13, 1993. 
  
\bibitem[C6]{C6} 
J.-M.~Coron,
  \emph{Control and nonlinearity}, Mathematical Surveys and Monographs, vol.~136, American Mathematical Society, Providence, RI, 2007.
  
  \bibitem[CKV]{CKV} 
G.~Camliyurt, I.~Kukavica, and V.~Vicol, 
\emph{Analyticity up to the boundary for the {S}tokes and the {N}avier-{S}tokes systems}, 
Trans.\ Amer.\ Math.\ Soc.~\textbf{373} (2020), no.~5, 3375--3422. 


\bibitem[CMS]{CMS} 
J.-M.~Coron, F.~Marbach, and F.~Sueur,
  \emph{Small-time global exact controllability of the {N}avier-{S}tokes equation
  with {N}avier slip-with-friction boundary conditions},
  J.~Eur.\ Math.\ Soc.~(JEMS)~\textbf{22} (2020), no.~5, 1625--1673.

\bibitem[CMSZ]{CMSZ} 
J.-M.~Coron, Fr\'{e}d\'{e}ric Marbach, F.~Sueur, and P.~Zhang,
  \emph{Controllability of the {N}avier-{S}tokes equation in a rectangle with a
  little help of a distributed phantom force}, Ann. PDE~\textbf{5} (2019),
  no.~2, Paper No.~17,~49.

\bibitem[CS]{CS} 
C. H. A.~Cheng, S.~Shkoller,
   \emph{Solvability and Regularity for an Elliptic System Prescribing the Curl, Divergence, and
Partial Trace of a Vector Field on Sobolev-Class Domains}, J. Math. Fluid Mech.~\textbf{19} (2017), 375–-422.

\bibitem[DL]{DL} 
R.~Dautray and J.-L.~Lions,
  \emph{Mathematical analysis and numerical methods for science and technology.\ {V}ol.~3},
  Springer-Verlag, Berlin, 1990, Spectral theory and applications, With the collaboration of
  Michel Artola and Michel Cessenat, Translated from the French by John C.~Amson.

\bibitem[E]{E}
W.M.~Elsasser,
  \emph{The hydromagnetic equations},
  Phys.\ Rev.~\textbf{79} (1950), 183--183.

\bibitem[FSS]{FSS} 
E.~Fern\'{a}ndez-Cara, M.C.~Santos, and D.A.~Souza,
  \emph{Boundary controllability of incompressible {E}uler fluids with
  {B}oussinesq heat effects}, Math. Control Signals Systems~\textbf{28} (2016), no.~1, Art. 7,~28.

\bibitem[G1]{G1} 
O.~Glass,
  \emph{Exact boundary controllability of 3-{D} {E}uler equation},
  ESAIM Control Optim.\ Calc.\ Var.~\textbf{5} (2000), 1--44.

\bibitem[G2]{G2} 
O.~Glass,
  \emph{An addendum to a {J}.{M}.~{C}oron theorem concerning the
  controllability of the {E}uler system for 2{D} incompressible inviscid
  fluids. ``{O}n the controllability of 2-{D} incompressible perfect fluids''
  [{J}. {M}ath. {P}ures {A}ppl. (9)~{\bf 75} (1996), no. 2, 155--188;
  {MR}1380673 (97b:93010)]},
  J.~Math. Pures Appl.~(9)~\textbf{80} (2001), no.~8, 845--877.
    
\bibitem[Ga]{Ga}
G.~P. Galdi.
\newblock {\em An introduction to the mathematical theory of the
  {N}avier-{S}tokes equations}.
\newblock Springer Monographs in Mathematics. Springer, New York, second
  edition, 2011.
\newblock Steady-state problems.    
    
    \bibitem[JKL1]{JKL1}
J.~Jang, I.~Kukavica, and L.~Li, 
\emph{Mach limits in analytic spaces},
J.~Differential Equations, \textbf{299} (2021), 284--332.

\bibitem[JKL2]{JKL2}
J.~Jang, I.~Kukavica, and L.~Li, 
\emph{Mach limits in analytic spaces on exterior domains},
Discrete Contin.\ Dynam.\ Systems~\textbf{42} (2022), 3629--3659.

\bibitem[K1]{K1} 
  G.~Komatsu,
  \emph{Analyticity up to the boundary of solutions of nonlinear parabolic equations},
  Comm. Pure Appl. Math. \textbf{32} (1979), no.~5, 669--720.

\bibitem[K2]{K2} 
G. Komatsu, \emph{Global analyticity up to the boundary of solutions of the
  {N}avier-{S}tokes equation}, Comm. Pure Appl. Math. \textbf{33} (1980),
  no.~4, 545--566.

\bibitem[KNV]{KNV} 
I.~Kukavica, M.~Novack, and V.~Vicol,
  \emph{Exact boundary controllability for the ideal magneto-hydrodynamic equations},
  J.~Differential Equations~\textbf{318} (2022), 94--112.

\bibitem[KOS]{KOS} 
I.~Kukavica, W.~S.~O\.za\'nski and M.~Sammartino,
  \emph{The inviscid inflow-outflow problem via analyticity},
  arXiv:2310.20439.



\bibitem[KP1]{KP1} 
S.G.~Krantz and H.R.~Parks,
  \emph{The implicit function theorem},
  Modern Birkh\"{a}user Classics, Birkh\"{a}user/Springer, New York, 2013,
  History, theory, and applications, Reprint of the 2003 edition.

\bibitem[KP2]{KP2} 
S.G.~Krantz and H.R.~Parks,
  \emph{A primer of real analytic functions},
  second ed., Birkh\"{a}user Advanced Texts: Basler Lehrb\"{u}cher.
  [Birkh\"{a}user Advanced Texts: Basel Textbooks], Birkh\"{a}user Boston,
  Inc., Boston, MA, 2002.

\bibitem[L1]{L1} 
J.-L. Lions,
  \emph{Exact controllability for distributed systems. {S}ome trends and some problems},
  Applied and industrial mathematics ({V}enice, 1989),
  Math.\ Appl., vol.~56, Kluwer Acad. Publ., Dordrecht, 1991, pp.~59--84.

\bibitem[L2]{L2}
J.L.~Lions,
  \emph{On the controllability of distributed systems},
  Proc.\ Nat.\ Acad.\ Sci.\ U.S.A.~\textbf{94} (1997), no.~10, 4828--4835.

\bibitem[L3]{L3}
J.L.~Lions,
  \emph{Are there connections between turbulence and controllability?}, $9^e$ {\em Conf\'erence internationale de ;'INRIA}, Antibes, 12--15 juin 1990. 

\bibitem[OP]{OP}
W.~S. O\.{z}a\'{n}ski and B.~C. Pooley.
\newblock Leray's fundamental work on the {N}avier-{S}tokes equations: a modern
  review of {\it ``{s}ur le mouvement d'un liquide visqueux emplissant
  l'espace''}.
\newblock In {\em Partial differential equations in fluid mechanics}, volume
  452 of {\em London Math. Soc. Lecture Note Ser.}, pages 113--203. Cambridge
  Univ. Press, Cambridge, 2018.
  
  
  \bibitem[P]{P} 
Ch.~Pommerenke,
  \emph{Boundary behaviour of conformal maps},
 Grundlehren der mathematischen Wissenschaften, vol.~299, Springer-Verlag, Berlin, 1992. 

\bibitem[R]{R}
  M.~Rissel,
  \emph{Exact controllability of incompressible ideal magnetohydrodynamics in 2D},
  arXiv:2306.03712.

\bibitem[RW1]{RW1} 
M.~Rissel and Y.-G.~Wang,
  \emph{Global exact controllability of ideal incompressible magnetohydrodynamic flows through a planar duct},
  ESAIM Control Optim.\ Calc.\ Var.~\textbf{27} (2021), Paper No. 103, 24.

\bibitem[RW2]{RW2}
M.~Rissel and Y.-G.~Wang,
  \emph{Small-time global approximate controllability for
        incompressible MHD with coupled Navier slip boundary conditions},
   arXiv:2203.10758.

\bibitem[S]{S} 
P.G.~Schmidt,
  \emph{On a magnetohydrodynamic problem of {E}uler type},
  J.~Differential Equations~\textbf{74} (1988), no.~2, 318--335.

\bibitem[Se]{Se} 
P.~Secchi,
  \emph{On the equations of ideal incompressible magnetohydrodynamics},
  Rend.\ Sem.\ Mat.\ Univ.\ Padova~\textbf{90} (1993), 103--119.

\bibitem[T]{T} 
R.~Temam,
  \emph{Navier-Stokes equations}, Theory and numerical analysis, Reprint of the 1984 edition, AMS Chelsea Publishing, Providence, RI, 2001. 

\bibitem[W]{W} 
J.~Wu,
  \emph{The 2{D} magnetohydrodynamic equations with partial or fractional dissipation},
  Lectures on the analysis of nonlinear partial differential equations.
  {P}art~5, Morningside Lect.\ Math., vol.~5, Int.\ Press, Somerville, MA, 2018, pp.~283--332.










\end{thebibliography}
\end{document}